\documentclass[11pt]{amsart}
\usepackage{amsthm,amssymb,mathrsfs,mathtools,empheq}
\usepackage[shortlabels]{enumitem}
\usepackage[T1]{fontenc}

\usepackage[notref,notcite,final]{showkeys}
\mathtoolsset{showonlyrefs}
\usepackage{fullpage,color}

\newtheorem{mainthm}{Theorem}
\newtheorem{theorem}{Theorem}[section]
\newtheorem*{theorem*}{Theorem}
\newtheorem{corollary}[theorem]{Corollary}

\newtheorem{lemma}[theorem]{Lemma}
\newtheorem{proposition}[theorem]{Proposition}
\newtheorem*{proposition*}{Proposition}

\newtheorem*{conjecture*}{Conjecture}

\theoremstyle{definition}

\newtheorem{remark}[theorem]{Remark}

\numberwithin{equation}{section}

\def\bC {\mathbb{C}}
\def\bN {\mathbb{N}}

\def\bR {\mathbb{R}}

\def\cE {\mathcal{E}}

\def\cS {\mathcal{S}}

\def\cY {\mathcal{Y}}

\def\scrL{\mathscr{L}}

\def\grad {{\nabla}}

\def\la {\langle}
\def\ra {\rangle}

\newcommand{\tx}[1]{\mathrm{#1}}
\newcommand{\wto}{\rightharpoonup}

\newcommand{\wt}[1]{\widetilde{#1}}

\newcommand{\conj}[1]{\overline{#1}}

\newcommand{\sign}{\operatorname{sign}}
\newcommand{\spn}{\operatorname{span}}

\newcommand{\dist}{\operatorname{dist}}
\renewcommand{\ker}{\operatorname{ker}}

\newcommand{\eee}{\mathrm e}

\newcommand{\ud}{\mathrm{\,d}}
\newcommand{\vd}{\mathrm{d}}

\newcommand{\vD}{\mathrm{D}}
\newcommand{\dd}[1]{{\frac{\vd}{\vd{#1}}}}

\newcommand{\lin}{_{\textsc{l}}}

\title{Construction of two-bubble solutions \\ for the energy-critical NLS}
\author{Jacek Jendrej}
\address{Department of Mathematics, University of Chicago, 5734 S. University Avenue, Chicago, IL 60637}
\email{jacek@math.uchicago.edu}

\begin{document}

\begin{abstract}
We construct pure two-bubbles for the energy-critical focusing nonlinear Schr\"odinger equation
in space dimension $N \geq 7$. The constructed solution is global in (at least) one time direction
and approaches a superposition of two stationary states both centered at the origin,
with the ratio of their length scales converging to $0$.
One of the bubbles develops at scale $1$, whereas the length scale of the other
converges to $0$ at rate $|t|^{-\frac{2}{N-6}}$.
The phases of the two bubbles form the right angle.
\end{abstract}

\maketitle
\section{Introduction}
\label{sec:intro}
\subsection{Setting of the problem}
\label{ssec:setting}
We consider the Schr\"odinger equation with the focusing energy-critical power nonlinearity:
\begin{equation}
\label{eq:nls}
i\partial_t u(t, x) + \Delta u(t, x) + f(u(t, x)) = 0, \qquad f(u) := |u|^\frac{4}{N-2}u, \qquad t \in \bR, x \in \bR^N.
\end{equation}
This equation can be studied in space dimension $N \geq 3$, but here we will restrict our attention to the case $N \geq 7$.

The \emph{energy functional} associated with this equation is defined for $u_0 \in \dot H^1(\bR^N; \bC)$ by the formula
\begin{equation*}
  E(u_0) := \int_{\bR^N} \frac 12|\grad u_0(x)|^2 - F(u_0(x))\ud x,
\end{equation*}
where $F(u_0) := \frac{N-2}{2N} |u_0|^\frac{2N}{N-2}$. Note that $E(u_0)$ is well-defined due to the Sobolev Embedding Theorem.
The differential of $E$ is $\vD E(u_0) = -\Delta u_0 - f(u_0)$, hence we have the following Hamiltonian form of the equation \eqref{eq:nls}:
\begin{equation}
  \label{eq:nlsH}
      \partial_t u(t) = -i \vD E(u(t)).
\end{equation}

  Equation \eqref{eq:nls} is locally well-posed in the space $\dot H^1(\bR^N)$, as was proved by Cazenave and Weissler \cite{CaWe90},
  see also a complete review of the Cauchy theory in \cite{KeMe06} (for $N \in \{3, 4, 5\}$) and \cite{KiVi10} (for $N \geq 6$).
  By ``well-posed'' we mean that for any initial data $u_0 \in \dot H^1(\bR^N)$ there exists $\tau > 0$
  and a linear subspace $S \subset C([t_0 - \tau, t_0 + \tau]; \dot H^1(\bR^N))$ such that there exists a unique weak solution $u(t) \in S$ of \eqref{eq:nls}
  satisfying $u(t_0) = u_0$, and that this solution is continuous with respect to the inital data.
  By standard arguments, there exists a maximal time of existence $(T_-, T_+)$, $-\infty \leq T_- < t_0 < T_+ \leq +\infty$,
  and a unique solution $u \in C((T_-, T_+); \dot H^1(\bR^N))$.
  Moreover, if $u_0 \in X^1 := \dot H^2(\bR^N) \cap \dot H^1(\bR^N)$, then $u \in C((T_-, T_+); X^1)$.
  If $T_+ < +\infty$, then $u(t)$ leaves every compact subset of $\dot H^1(\bR^N)$ as $t$ approaches $T_+$.
  A crucial property of the solutions of \eqref{eq:nls} is that the energy $E$ is a conservation law.
  If $u_0 \in L^2$, then the mass $\|u(t)\|_{L^2}^2$ is another conservation law, but we will never use this fact.

  In this paper, we always assume that the initial data are radially symmetric. This symmetry is preserved by the flow. We denote $\cE$ the space radially symmetric functions in $\dot H^1(\bR^N; \bC)$.

  For a function $v \in \cE$, we denote
  \begin{equation*}
    v_\lambda(x) := \frac{1}{\lambda^\frac{N-2}{2N}} v\big(\frac{x}{\lambda}\big).
  \end{equation*}
  A change of variables shows that
  \begin{equation*}
    E\big((u_0)_\lambda\big) = E(u_0).
  \end{equation*}
  Equation~\eqref{eq:nls} is invariant under the same scaling: if $u(t)$ is a solution of \eqref{eq:nls} and $\lambda > 0$, then
  $
  t \mapsto u\big(t_0 + \lambda^{-2}t\big)_\lambda
  $ is also a solution
  with initial data $(u_0)_\lambda$ at time $t = 0$.
  This is why equation~\eqref{eq:nls} is called \emph{energy-critical}.

  The solutions of the corresponding \emph{defocusing} equation exist globally and scatter. This was proved by Bourgain \cite{Bourgain99} and Tao \cite{Tao05}
for radial solutions, and by Colliander, Keel, Staffilani, Takaoka and Tao \cite{Iteam08}, Ryckman and Visan \cite{RyVi07}, and Visan \cite{Visan07}
for non-radial data.

  The study of the dynamical behavior of solutions of the focusing equation \eqref{eq:nls} for large initial data was initiated by Kenig and Merle \cite{KeMe06}.
  In this case, an important role is played by the family of stationary solutions $u(t) \equiv \eee^{i\theta} W_\lambda$, where
  \begin{equation*}
    W(x) = \Big(1 + \frac{|x|^2}{N(N-2)}\Big)^{-\frac{N-2}{2}}.
  \end{equation*}
  The functions $\eee^{i\theta}W_\lambda$ are called \emph{ground states} or \emph{bubbles} (of energy). They are the only radially symmetric solutions
  of the critical elliptic problem
  \begin{equation}
    \label{eq:elliptic}
    -\Delta u - f(u) = 0.
  \end{equation}
  The ground states achieve the optimal constant in the critical Sobolev inequality, which was proved by Aubin \cite{Aubin76} and Talenti \cite{Talenti76}.
  They are the ``mountain passes'' for the potential energy.

  Kenig and Merle \cite{KeMe06} exhibited the special role of the ground states $\eee^{i\theta}W_\lambda$ as the \emph{threshold elements} for nonlinear dynamics of the solutions of \eqref{eq:nls}
  in space dimensions $N = 3, 4, 5$ for radial data. They proved the so-called \emph{Threshold Conjecture} by completely classifying the dynamical behavior of solutions $u(t)$ of \eqref{eq:nls}
  such that $E(u(t)) < E(W)$. An analogous result in higher dimensions, for non-radial data, was obtained by Killip and Visan \cite{KiVi10}.

  A much stronger statement about the dynamics of solutions is the \emph{Soliton Resolution Conjecture}, which predicts that a bounded (in an appropriate sense) solution
  decomposes asymptotically into a sum of energy bubbles at different scales and a radiation term (a solution of the linear Schr\"odinger equation).
  This was proved for the radial energy-critical wave equation in dimension $N = 3$
  by Duyckaerts, Kenig and Merle \cite{DKM4}, see also \cite{DJKM} for the non-radial case. For \eqref{eq:nls} this problem is completely open.

  Solutions slightly above the ground state energy threshold were studied by Ortoleva and Perelman~\cite{OrPe13} in dimension $N = 3$,
  see also Perelman~\cite{Perelman14} for the closely related critical equivariant Schr\"odinger maps equations with values in the sphere.
  They constructed global solutions which stay close to $\eee^{i\theta}W_\lambda$ in the energy space, with $\lambda$ converging to $0$ as time $t$ goes to $+\infty$.
  These solutions decompose into a concentrating bubble and a radiation term, in accordance with the Soliton Resolution Conjecture.
  The works of Ortoleva and Perelman follow the approach developed by Krieger, Schlag and Tataru \cite{KrScTa08, KrScTa09} for wave equations.
  For the Schr\"odinger maps, following a different approach, Merle, Rodnianski and Rapha\"el \cite{MeRaRo13} obtained blow-up solutions
  which are stable relative to a set of finite codimension in some space which contains the bubble.

  On the classification side, it is unknown whether the Soliton Resolution Conjecture holds even with an additional assumption that the solution
  remains close to the family of the ground states. In the mass-critical case and for a solution blowing up in finite time,
  this was proved by Merle and Rapha\"el \cite{MeRa05, MeRa05-2}, see also Fan \cite{Chenjie16p}.
\subsection{Main results}
  In view of the Soliton Resolution Conjecture, solutions which exhibit no dispersion in one or both time directions play a distinguished role.
  One obvious example of such solutions are the static solutions $\eee^{i\theta}W_\lambda$.
  In this paper, we consider the simplest non-trivial case, namely
  we construct global radial solutions which approach, in the energy space,
  a sum of two bubbles. The ratio of the scales at which these bubbles develop tends to $0$.

\begin{mainthm}
  \label{thm:deux-bulles}
There exists a solution $u: (-\infty, T_0] \to \cE$ of \eqref{eq:nls} such that
  \begin{equation}
    \label{eq:mainthm}
    \lim_{t\to -\infty}\Big\|u(t) - \Big({-}iW + W_{\frac{1}{\kappa}(\kappa|t|)^{-\frac{2}{N-6}}}\Big)\Big\|_\cE = 0,
  \end{equation}
  where $\kappa$ is an explicit constant.
\end{mainthm}
\begin{remark}
  For the value of $\kappa$, see \eqref{eq:kappa}.
\end{remark}
\begin{remark}
  More precisely, we will prove that
  \begin{equation*}
    \Big\|u(t) - \Big({-}iW + W_{\frac{1}{\kappa}(\kappa|t|)^{-\frac{2}{N-6}}}\Big)\Big\|_\cE \leq C_1|t|^{-\frac{1}{2(N-6)}}
  \end{equation*}
  for some constant $C_1 > 0$.
\end{remark}
\begin{remark}
We construct here \emph{pure} two-bubbles, that is the solution approaches a superposition of two stationary states, with no energy transformed into radiation.
By the conservation of energy and the decoupling of the two bubbles, we necessarily have $E(u(t)) = 2E(W)$.
Pure one-bubbles cannot concentrate and are completely classified, see \cite{DM09}.
\end{remark}
\begin{remark}
  For energy-critical wave equations, similar objects were constructed in \cite{moi16p}.
\end{remark}
\begin{remark}
  \label{rem:dim7}
  In dimension $N = 6$ one can expect an analogous result, with an exponential concentration rate.
\end{remark}
\begin{remark}
  In higher dimension, fast dispersion or dissipation sometimes excludes the possibility of a concentration of a bubble
  of energy for solutions which belong to a small neighborhood of a bubble.
  This was proved in \cite{CoMeRa16p} in the case of the critical heat equation,
  see also \cite{Perelman16sem} for the Schr\"odinger equation.
  We prove here that once we leave a small neighborhood of a bubble,
  concentration of a bubble of energy is possible in arbitrarily high dimension.
\end{remark}
\begin{remark}
  I expect that the phases of the two bubbles forming the right angle is the only configuration
  in which a two-bubble can form. 
\end{remark}

\subsection{Outline of the proof}
The overall structure is similar as in the earlier work of the author on the critical wave equations \cite{moi16p}.
We build a sequence $u_n: [T_n, T_0] \to \cE$ of solutions of \eqref{eq:nls}
with $T_n \to -\infty$ and $u_n(t)$ close to a two-bubble solution for $t \in [T_n, T_0]$.
Taking a weak limit finishes the proof. This type of argument goes back to the works of Merle~\cite{Merle90} and Martel~\cite{Martel05}.
The heart of the analysis is to obtain uniform energy bounds for the sequence $u_n$.
To do this, we use a new, simplified approach. It can be resumed as follows.

We study solutions of \eqref{eq:nls} close to a sum of two bubbles:
\begin{equation}
  u(t) = \eee^{i\zeta(t)}W_\mu + \eee^{i\theta(t)}W_{\lambda(t)} + g(t).
\end{equation}
One should think of $\zeta(t)$ as being close to $-\frac{\pi}{2}$, $\mu(t) \simeq 1$,
$\theta(t) \sim 0$, $\lambda(t) \ll 1$ and $\|g(t)\|_\cE \ll 1$.
In order to specify the values of the modulation parameters,
we impose the orthogonality conditions which make disappear terms linear in $g$ in the modulation equations.
There is essentially a unique choice of such orthogonality conditions.
In Lemma~\ref{lem:mod} we establish bounds on the evolution of the modulation parameters under some bootstrap assumptions.
The goal is to improve these bounds, thus closing the bootstrap.
The essential point is to improve the estimate of $g$, which is the inifinite-dimensional part.
The novelty of this paper is to use the energy conservation to deal with this.
Namely, the energy of the initial data is chosen close to $2E(W)$ and is conserved by the flow.
It turns out that if we control the modulation parameters sufficiently well, we can improve the bound on $\|g\|_\cE$
by simply expanding the formula for $E(u)$ and using coercivity of the energy near a ground state, see Step 3 of the proof of Proposition~\ref{prop:bootstrap}.

It remains to control the modulations parameters. Note that the interaction between the two bubbles appears explicitely in the modulation equation for $\lambda'(t)$,
see \eqref{eq:mod-l}. In fact, the configuration of the two bubbles (phases forming the right angle)
is chosen so as to \emph{maximize} the size of the term appearing in \eqref{eq:mod-l}
and leading to the growth of the parameter $\lambda$. The critical part of the proof consists in improving the bound \eqref{eq:bootstrap-theta} on $\theta(t)$.
To this end, we add a localized virial correction to $\theta(t)$ to cancel the main quadratic, which is $K(t)$ in the modulation equation \eqref{eq:mod-th}.
Note that the size of the term $\frac{K(t)}{\lambda(t)^2\|W\|_{L^2}^2}$ in \eqref{eq:mod-th} is $O(|t|^{-\frac{N-5}{N-6}})$.
Adding the virial correction allows us to gain a small constant on the right hand side of \eqref{eq:mod-th}, which is decisive for closing the bootstrap.

Finally, in order to deal with the linear instabilities of the flow,
we use a classical topological argument based on the Brouwer fixed point theorem.

\subsection{Acknowledgments}
Part of this work was realized when I was a PhD student at \'Ecole polytechnique.
I would like to thank my advisors Yvan Martel and Frank Merle for encouraging me to pursue this project.
I was partially supported by the ERC grant 291214 BLOWDISOL.

\subsection{Notation}
For $z = x + iy \in \bC$ we denote $\Re(z) = x$ and $\Im(z) = y$. For two functions $v, w \in L^2(\bR^N, \bC)$ we denote
\begin{equation}
  \la v, w\ra := \Re\int_{\bR^N} \conj{v(x)}\cdot w(x)\ud x.
\end{equation}
In this paper all the functions are radially symmetric. We write $L^2 := L^2_{\tx{rad}}(\bR^N; \bC)$ and $\cE := \dot H^1_{\tx{rad}}(\bR^N; \bC)$.
We will think of them as of \emph{real} vector spaces.
We denote $X^1 := \cE \cap \dot H^2(\bR^N)$.

\section{Variational estimates}
\label{sec:variational}
\subsection{Linearization near a ground state}
Recall that for $u \in \bC$ we denote $f(u) := |u|^\frac{4}{N-2}u$ and $F(u) := \frac{N-2}{2N}|u|^\frac{2N}{N-2}$.
For $u \in \bC$ we define the $\bR$-linear function $f'(u): \bC \to \bC$ by the following formula:
\begin{equation}
  f'(u)g := |u|^\frac{4}{N-2}\Big(g + \frac{4}{N-2}u\Re(u^{-1}g)\Big)
\end{equation}
(with the convention $f'(0)g = 0$). It is easy to check that for any $g, h, u \in \bC$ there holds
\begin{equation}
  \label{eq:auto-scalar}
  \Re\big(\conj h(f'(u)g)\big) = \Re \big(\conj g(f'(u)h)\big) = \Re\big((\conj{f'(u)h})g\big).
\end{equation}
Integrating this identity on $\bR^N$ we see that for a complex function $u(x)$ the operator $g \mapsto f'(u)g$ is symmetric with respect to the real $L^2$ scalar product.
We denote $|f'(u)| := \frac{N+2}{N-2}|u|^\frac{4}{N-2}$, which is the norm of $f'(u)$ as a linear map up to a constant.
For $u: \bR^N \to \bC$ we define $\|f'(u)\|_{L^p} := \big(\int_{\bR^N}|f'(u(x))|^p\ud x\big)^\frac 1p$ for $1 \leq p < +\infty$ and $\|f'(u)\|_{L^\infty} := \sup_{x \in \bR^N}|f'(u(x))|$.
\begin{lemma}
  \label{lem:pointwise}
  Let $N \geq 7$. For $z_1, z_2, z_3 \in \bC$ there holds
  \begin{gather}
    |f'(z_1 + z_2) - f'(z_1)| \lesssim |f'(z_2)|,\qquad |f'(z_1 + z_2) - f'(z_1)| \lesssim |z_1|^{-\frac{N-6}{N-2}}|z_2|\text{ if }z_1 \neq 0, \label{eq:pointwise-5} \\
    |f(z_1 + z_2) - f(z_1)| \lesssim |f'(z_1)|\cdot|z_2| + |f(z_2)|, \label{eq:pointwise-2} \\
    \begin{gathered}
      |f(z_1 + z_2) - f(z_1) - f'(z_1)z_2| \lesssim f(|z_2|) \qquad \text{and }\\ |f(z_1 + z_2) - f(z_1) - f'(z_1)z_2| \lesssim |z_1|^{-\frac{N-6}{N-2}}|z_2|^2\quad \text{ if }z_1\neq 0,
  \end{gathered} \label{eq:pointwise-1} \\
    \big|F(z_1 + z_2) - F(z_1) - \Re\big(\conj{f(z_1)}\cdot z_2\big)\big| \lesssim |f'(z_1)|\cdot|z_2|^2 + F(z_2), \label{eq:pointwise-6} \\
    \big|F(z_1 + z_2) - F(z_1) - \Re\big(\conj{f(z_1)}\cdot z_2\big) - \Re\big(\conj{f'(z_1)z_2}\cdot z_2\big)\big| \lesssim F(z_2). \label{eq:pointwise-3}
  \end{gather}
\end{lemma}
\begin{remark}
  In \eqref{eq:pointwise-5}, $|f'(z_1+ z_2) - f'(z_1)|$ denotes the norm of $f'(z_1 + z_2) - f'(z_1)$ as an $\bR$-linear map.
\end{remark}
\begin{remark}
Note that \eqref{eq:pointwise-1} implies that $f'(u)$ is the derivative (in the real sense) of $f$ at $u$,
in particular $f$ is a $C^1$ function.
\end{remark}
\begin{proof}
All the bounds are immediate if $|z_2| \geq \frac 12 |z_1|$, hence we can assume that $|z_2| < \frac 12 |z_1|$, in particular $z_1 \neq 0$.

The formulas $f'(z_1)z_2 = f(z_1)f'(1)(z_1^{-1}z_2)$ and $f'(z_1+z_2)z_3 = f(z_1)f'(1+z_1^{-1}z_2)(z_1^{-1}z_3)$
allow to reduce the proof to the case $z_1 = 1$. For $|z| < \frac 12$, the mappings $F(1 + z)$, $f(1 + z)$ and $f'(1 + z)$ are real-analytic with respect to $z$
and the required bounds follow by writing standard asymptotic expansions.


\end{proof}
We denote $Z_{\theta, \lambda} := i\Delta + if'(\eee^{i\theta}W_\lambda)$ the linearization of $i\Delta u + if(u)$ near $u = \eee^{i\theta}W_\lambda$.
In order to express $Z_{\theta, \lambda}$ in a more explicit way, we introduce the following notation:
\begin{equation}
V^+ := -\frac{N+2}{N-2}W^\frac{4}{N-2}, \qquad V^- := -W^\frac{4}{N-2}, \qquad L^+ := -\Delta +V^+, \qquad L^- := -\Delta + V^-.
\end{equation}
It is known that for all $g \in \cE$ there holds $\la g, L^- g\ra \geq 0$ and $\ker L^- = \spn(W)$.
The operator $L^+$ has one simple strictly negative eigenvalue and,
restricting to radially symmetric functions, $\ker L^+ = \spn(\Lambda W)$.

For future reference, we provide here the values of some integrals involving $W$ and $\Lambda W$:
    \begin{align}
      \int_{\bR^N} W^2 \ud x &= \frac 12 \big(N(N-2)\big)^\frac N2 B\Big(\frac{N-4}{2}, \frac N2\Big), \label{eq:explicit-1} \\
      \int_{\bR^N} W^\frac{N+2}{N-2} &= \frac 1N \big(N(N-2)\big)^\frac N2, \label{eq:explicit-2} \\
      -\frac{N+2}{N-2}\int_{\bR^N} W^\frac{4}{N-2}\Lambda W \ud x &= \frac{N-2}{2N}\big(N(N-2)\big)^\frac N2 \label{eq:explicit-3}.
    \end{align}
For the first integral, we use the formula $B(x, y) = \int_0^{+\infty}t^{x-1}(1+t)^{-x-y}\ud t$.
For the second, we write $W^\frac{N+2}{N-2} = -\Delta W$ and we integrate by parts.
For the last integral, we write $-\frac{N+2}{N-2}W^\frac{4}{N-2}\Lambda W = V^+\Lambda W = \Delta \Lambda W$ and we integrate by parts.

Using the definition of $f'$, one can check that if $g_1 = \Re g$ and $g_2 = \Im g$, then
\begin{equation}
  Z_{\theta, \lambda}(\eee^{i\theta}g_\lambda) = \frac{\eee^{i\theta}}{\lambda^2}(L^-g_2 - iL^+ g_1)_\lambda.
\end{equation}
In particular, we obtain
\begin{gather}
  Z_{\theta, \lambda}(i\eee^{i\theta}W_\lambda) = \frac{\eee^{i\theta}}{\lambda^2}(L^- W)_\lambda = 0, \\
  Z_{\theta, \lambda}(\eee^{i\theta}\Lambda W_\lambda) = \frac{\eee^{i\theta}}{\lambda^2}(-iL^+ \Lambda W)_\lambda = 0.
\end{gather}
This can also be seen by differentiating $i\Delta (\eee^{i\theta}W_\lambda) + if(\eee^{i\theta}W_\lambda)$
with respect to $\theta$ and $\lambda$.

Consider now the operator $Z_{\theta, \lambda}^*$.
We claim that $\{\eee^{i\theta}W_\lambda, i\eee^{i\theta}\Lambda W_\lambda\} \subset \ker Z_{\theta, \lambda}^*$. Indeed, we have
  \begin{gather}
    \la \eee^{i\theta}W_\lambda, Z_{\theta, \lambda}(\eee^{i\theta}g_\lambda)\ra = \big\la \eee^{i\theta}W_\lambda, \frac{\eee^{i\theta}}{\lambda^2}(L^- g_2 - iL^+g_1)_\lambda\big\ra
    = \la W, L^- g_2\ra = \la L^- W, g_2\ra = 0,
  \label{eq:conj-ker-1} \\
  \begin{aligned}
  \la i\eee^{i\theta}\Lambda W_\lambda, Z_{\theta, \lambda}(\eee^{i\theta}g_\lambda)\ra &= \big\la \eee^{i\theta}\Lambda W_\lambda, \frac{\eee^{i\theta}}{\lambda^2}(L^- g_2 - iL^+g_1)_\lambda\ra \\
  &= -\la \Lambda W, L^+ g_1\ra = -\la L^+ \Lambda W, g_1\ra = 0.
\end{aligned}
  \label{eq:conj-ker-2}
  \end{gather}

%
One can show that there exist real functions $\cY^{(1)}, \cY^{(2)} \in \cS$ and a real number $\nu > 0$ such that
\begin{equation}
  \label{eq:Y1Y2}
  L^+ \cY^{(1)} = -\nu \cY^{(2)}, \qquad L^- \cY^{(2)} = \nu \cY^{(1)}
\end{equation}
(the proof given in \cite[Section 7]{DM09} for $N = 5$ works in any dimension $N \geq 5$).
We can assume that $\|\cY^{(1)}\|_{L^2} = \|\cY^{(2)}\|_{L^2} = 1$. We denote
\begin{equation}
  \label{eq:alpha}
  \alpha_{\theta, \lambda}^+ := \frac{\eee^{i\theta}}{\lambda^2}\big(\cY_\lambda^{(2)} + i\cY_\lambda^{(1)}\big), \qquad \alpha_{\theta, \lambda}^- := \frac{\eee^{i\theta}}{\lambda^2}\big(\cY_\lambda^{(2)} - i\cY_\lambda^{(1)}\big).
\end{equation}
For $g = g_1 + ig_2$ we have $\la \alpha_{\theta, \lambda}^+, \eee^{i\theta}g_\lambda\ra = \la \cY^{(2)}, g_1\ra + \la \cY^{(1)}, g_2\ra$
and $\la \alpha_{\theta, \lambda}^-, \eee^{i\theta}g_\lambda\ra = \la \cY^{(2)}, g_1\ra - \la \cY^{(1)}, g_2\ra$.
Note that
\begin{gather}
\la W, \cY^{(1)}\ra = \frac{1}{\nu}\la W, L^- \cY^{(2)}\ra = \frac{1}{\nu}\la L^- W, \cY^{(2)}\ra = 0, \\
\la \Lambda W, \cY^{(2)}\ra = -\frac{1}{\nu}\la \Lambda W, L^+ \cY^{(1)}\ra = -\frac{1}{\nu}\la L^+(\Lambda W), \cY^{(1)}\ra = 0.
\end{gather}
It follows that
\begin{gather}
  \la \alpha_{\theta, \lambda}^+, i\eee^{i\theta}W_\lambda \ra = \la \alpha_{\theta, \lambda}^-, i\eee^{i\theta}W_\lambda\ra = 0, \label{eq:proper-iW} \\
  \la \alpha_{\theta, \lambda}^+, \eee^{i\theta}\Lambda W_\lambda \ra = \la \alpha_{\theta, \lambda}^-, \eee^{i\theta}\Lambda W_\lambda\ra = 0. \label{eq:proper-LW}
\end{gather}
Since $\cY^{(2)} \neq W$, we also have
\begin{equation}
\label{eq:Y1Y2-prod}
  \la \cY^{(1)}, \cY^{(2)}\ra = \frac{1}{\nu}\la \cY^{(2)}, L^-\cY^{(2)}\ra > 0.
\end{equation}

We claim that $\alpha_{\theta, \lambda}^+$ and $\alpha_{\theta, \lambda}^-$ are eigenfunctions of $Z_{\theta, \lambda}^*$,
with eigenvalues $\frac{\nu}{\lambda^2}$ and $-\frac{\nu}{\lambda^2}$ respectively. Indeed, we have
\begin{equation}
  \label{eq:ap-eigen}
  \begin{gathered}
    \la \alpha_{\theta, \lambda}^+, Z_{\theta, \lambda}(\eee^{i\theta}g_\lambda)\ra = \big\la\alpha_{\theta, \lambda}^+,
    \frac{\eee^{i\theta}}{\lambda^2}(L^-g_2 - iL^+ g_1)_\lambda\big\ra = \frac{1}{\lambda^2}(\la \cY^{(2)}, L^- g_2\ra - \la \cY^{(1)}, L^+ g_1\ra) \\
    = \frac{1}{\lambda^2}(\la L^-\cY^{(2)}, g_2\ra - \la L^+\cY^{(1)}, g_1\ra) = \frac{\nu}{\lambda^2}(\la \cY^{(1)}, g_2\ra + \la \cY^{(2)}, g_1\ra) = \frac{\nu}{\lambda^2}\la \alpha_{\theta, \lambda}^+, \eee^{i\theta}g_\lambda\ra.
  \end{gathered}
\end{equation}
Similarly, $\la \alpha_{\theta, \lambda}^-, Z_{\theta, \lambda}(\eee^{i\theta}g_\lambda)\ra = -\frac{\nu}{\lambda^2}\la \alpha_{\theta, \lambda}^-, \eee^{i\theta}g_\lambda\ra$.

\subsection{Coercivity of the energy near a two-bubble}
\label{ssec:coer-en}
We consider $u \in \cE$ of the form $u = \eee^{i\zeta}W_\mu + \eee^{i\theta}W_\lambda + g$
with
\begin{equation}
  \big|\zeta + \frac{\pi}{2}\big| + |\mu - 1| + |\theta| + \lambda + \|g\|_\cE \ll 1.
\end{equation}
Moreover, we will assume that $g$ satisfies
\begin{equation}
  \label{eq:orth}
  \la i\eee^{i\zeta}\Lambda W_\mu, g\ra = \la -\eee^{i\zeta}W_\mu, g\ra = \la i\eee^{i\theta}\Lambda W_{\lambda}, g\ra = \la -\eee^{i\theta}W_\lambda, g\ra =  0.
\end{equation}
This choice of the orthogonality conditions is dictated by the kernel of $Z_{\theta, \lambda}^*$, see \eqref{eq:conj-ker-1} and \eqref{eq:conj-ker-2}.
In this section this has little importance, but will be crucial in the sequel.

When $\zeta, \mu, \theta, \lambda$ and $g$ are known from the context, we denote
\begin{equation}
a_1^+ := \la \alpha_{\zeta, \mu}^+, g\ra, \qquad a_1^- := \la \alpha_{\zeta, \mu}^-, g\ra,\qquad a_2^+ := \la \alpha_{\theta, \lambda}^+, g\ra, \qquad a_2^- := \la \alpha_{\theta, \lambda}^-, g\ra.
\end{equation}

Our objective to prove the following result.
\begin{proposition}
  \label{prop:coercivity}
  There exist constants $\eta, C_0, C > 0$ depending only on $N$ such that for all $u \in \cE$ of the form $u = \eee^{i\zeta}W_\mu + \eee^{i\theta}W_\lambda + g$,
  with $\big|\zeta + \frac{\pi}{2}\big| + |\mu - 1| + |\theta| + \lambda + \|g\|_\cE \leq \eta$ and $g$ verifying \eqref{eq:orth}, there holds
  \begin{gather}
    \label{eq:coer-bound}
    |E(u) - 2E(W)| \leq C\Big(\big(\big|\zeta + \frac{\pi}{2}\big| + |\mu - 1| + |\theta| + \lambda\big)\lambda^\frac{N-2}{2} + \|g\|_\cE^2 \Big), \\
      \label{eq:coer-conclusion}
      \begin{aligned}
        \|g\|_\cE^2 + C_0 \theta\lambda^\frac{N-2}{2} &\leq C\Big(\lambda^\frac{N-2}{2}\big(\big|\zeta+\frac{\pi}{2}\big| + |\mu - 1|
        + |\theta|^3 + \lambda\big) \\
      &+ E(u) - 2E(W) + \sum_{j = 1, 2}\big((a_j^+)^2 + (a_j^-)^2\big)\Big).
    \end{aligned}
    \end{gather}
\end{proposition}
The scheme of the proof is the following. The inequality \eqref{eq:pointwise-3} yields the Taylor expansion of the energy:
\begin{equation}
  \label{eq:energy-taylor}
  \Big|E(u) - E(\eee^{i\zeta}W_\mu + \eee^{i\theta}W_\lambda) - \la \vD E(\eee^{i\zeta}W_\mu + \eee^{i\theta}W_\lambda), g\ra -
  \frac 12 \la \vD^2 E(\eee^{i\zeta}W_\mu + \eee^{i\theta}W_\lambda)g, g\ra\Big| \lesssim \|g\|_\cE^\frac{2N}{N-2}.
\end{equation}
We just have to compute all the terms with a sufficiently high precision.
We split this computation into a few lemmas.
\begin{lemma}
  \label{lem:coer-sans-g}
  Let $\zeta, \mu, \theta, \lambda$ be as in Proposition~\ref{prop:coercivity}. Then
  \begin{equation}
    \label{eq:coer-sans-g}
    \Big|E(\eee^{i\zeta}W_\mu + \eee^{i\theta}W_\lambda) - 2E(W) - \frac 1N \big(N(N-2)\big)^\frac N2 \theta\lambda^\frac{N-2}{2}\Big|
    \leq C\lambda^\frac{N-2}{2}\big(\big|\zeta + \frac{\pi}{2}\big| + |\mu - 1| + |\theta|^3 + \lambda\big),
  \end{equation}
  with a constant $C$ depending only on $N$.
\end{lemma}
\begin{proof}
  Expanding the energy we find
  \begin{equation}
    \label{eq:energy-expansion}
    \begin{aligned}
      E(\eee^{i\zeta}W_\mu + \eee^{i\theta}W_\lambda) &= E(\eee^{i\zeta}W_\mu) + E(\eee^{i\theta}W_\lambda) + \Re\int_{\bR^N}\eee^{i(\zeta-\theta)}\grad(W_\mu)\cdot\grad(W_\lambda)\ud x \\
      &- \int_{\bR^N}\big(F(\eee^{i\zeta}W_\mu + \eee^{i\theta}W_\lambda) - F(\eee^{i\zeta}W_\mu) - F(\eee^{i\theta}W_\lambda)\big)\ud x.
    \end{aligned}
  \end{equation}
  By scaling invariance, $E(\eee^{i\zeta}W_\mu) + E(\eee^{i\theta}W_\lambda) = 2E(W)$. Integrating by parts we get
  \begin{equation}
    \Re\int_{\bR^N}\eee^{i(\zeta-\theta)}\grad(W_\mu)\cdot\grad(W_\lambda)\ud x = -\Re \int_{\bR^N} \conj{\eee^{i\theta} W_\lambda}\Delta (\eee^{i\zeta}W_\mu)\ud x = \Re \int_{\bR^N} \conj{\eee^{i\theta} W_\lambda}\cdot f(\eee^{i\zeta}W_\mu)\ud x,
  \end{equation}
  hence \eqref{eq:energy-expansion} yields
  \begin{equation}
    \label{eq:energy-expansion-1}
    \begin{aligned}
    &E(\eee^{i\zeta}W_\mu + \eee^{i\theta}W_\lambda) = 2E(W) \\
    &- \int_{\bR^N}\big(F(\eee^{i\zeta}W_\mu + \eee^{i\theta}W_\lambda) - F(\eee^{i\zeta}W_\mu) - F(\eee^{i\theta}W_\lambda)-\Re\big( \conj{\eee^{i\theta} W_\lambda}\cdot f(\eee^{i\zeta}W_\mu)\big)\big)\ud x.
  \end{aligned}
  \end{equation}

  In the region $|x| \geq \sqrt\lambda$, 
  using \eqref{eq:pointwise-6} with $z_1 = \eee^{i\zeta}W_\mu$
  and $z_2 = \eee^{i\theta}W_\lambda$, we obtain
  \begin{equation}
    \big|F(\eee^{i\zeta}W_\mu + \eee^{i\theta}W_\lambda) - F(\eee^{i\zeta}W_\mu) - F(\eee^{i\theta}W_\lambda)-\Re\big( \conj{\eee^{i\theta} W_\lambda}\cdot f(\eee^{i\zeta}W_\mu)\big)\big| \lesssim W_\lambda^2,
  \end{equation}
  and we see that
\begin{equation}
  \int_{|x| \geq \sqrt\lambda}W_\lambda^2 = \lambda^2 \int_{|x| \geq 1/\sqrt\lambda}W^2 \ud x \lesssim \lambda^2 \int_{1/\sqrt\lambda}^{+\infty}r^{-2N+4}r^{N-1}\ud r = \lambda^{2 + \frac{N-4}{2}} = \lambda^\frac N2.
\end{equation}
In the region $|x| \leq \sqrt\lambda$ the last term in \eqref{eq:energy-expansion-1} is negligible,
  because $\big|\Re\big(\conj{\eee^{i\theta}W_\lambda}\cdot f(\eee^{i\zeta}W_\mu)\big)\big| \lesssim W_\lambda$
  and $\int_{|x|\leq \sqrt\lambda} W_\lambda \ud x \lesssim \lambda^\frac{N+2}{2}\int_0^{1/\sqrt\lambda}r^{-N+2}r^{N-1}\ud r \sim \lambda^\frac N2$. Similarly, the term $F(\eee^{i\zeta}W_\mu)$ is negligible.
 Using \eqref{eq:pointwise-6} with $z_1 = \eee^{i\theta}W_\lambda$ and $z_2 = \eee^{i\zeta}W_\mu$, we obtain
  \begin{equation}
    \big|F(\eee^{i\zeta}W_\mu + \eee^{i\theta}W_\lambda) - F(\eee^{i\theta}W_\lambda)-\Re\big( \conj{\eee^{i\zeta} W_\mu}\cdot f(\eee^{i\theta}W_\lambda)\big)\big| \lesssim W_\lambda^\frac{4}{N-2},
  \end{equation}
and we see that
\begin{equation}
  \int_{|x| \leq \sqrt\lambda}W_\lambda^\frac{4}{N-2} = \lambda^{N-2} \int_{|x| \leq 1/\sqrt\lambda}W^\frac{4}{N-2} \ud x \lesssim \lambda^{N-2} \int_{1/\sqrt\lambda}^{+\infty}r^{-4}r^{N-1}\ud r = \lambda^{N-2 - \frac{N-4}{2}} = \lambda^\frac N2.
\end{equation}
In order to complete the proof of \eqref{eq:coer-sans-g}, we thus need to check that
\begin{equation}
  \label{eq:energy-expansion-2}
  \begin{aligned}
  &\Big|{-}\int_{|x|\leq \sqrt\lambda}\Re\big( \conj{\eee^{i\zeta} W_\mu}\cdot f(\eee^{i\theta}W_\lambda)\big)\ud x - \frac 1N (N(N-2))^\frac N2 \theta\lambda^\frac{N-2}{2}\Big| \\
  &\lesssim \lambda^\frac{N-2}{2}\big(\big|\zeta + \frac{\pi}{2}\big| + |\mu - 1| + |\theta|^3 + \lambda\big).
\end{aligned}
\end{equation}
There holds
\begin{equation}
  \begin{aligned}
  &\Big|\int_{|x|\leq \sqrt\lambda}\Re\big( \conj{\eee^{i\zeta} W_\mu}\cdot f(\eee^{i\theta}W_\lambda)\big)\ud x -
  \Re\big(\eee^{i(\zeta-\theta)}\big)\int_{\bR^N}W_\lambda^\frac{N+2}{N-2}\ud x\Big|  \\
  &\lesssim \int_{|x| \leq \sqrt\lambda}|W_\mu - 1|W_\lambda^\frac{N+2}{N-2}\ud x + \int_{|x|\geq \sqrt\lambda}W_\lambda^\frac{N+2}{N-2}\ud x \\
  &\lesssim (|\mu - 1| + \lambda)\int_{|x| \leq \sqrt\lambda}W_\lambda^\frac{N+2}{N-2} + \int_{|x| \geq \sqrt\lambda}W_\lambda^\frac{N+2}{N-2}\ud x \\
  &\lesssim (|\mu - 1| + \lambda)\lambda^\frac{N-2}{2} + \lambda^\frac{N-2}{2} \int_{|x| \geq 1/\sqrt\lambda}W^\frac{N+2}{N-2}\ud x \lesssim (|\mu - 1| + \lambda)\lambda^\frac{N-2}{2}
\end{aligned}
\end{equation}
and
\begin{equation}
  \int_{\bR^N}W_\lambda^\frac{N+2}{N-2}\ud x = \lambda^\frac{N-2}{2} \int_{\bR^N}W^\frac{N+2}{N-2}\ud x = \frac 1N\big(N(N-2)\big)^\frac N2\lambda^\frac{N-2}{2}.
\end{equation}
We have
\begin{equation}
  \big|\Re(-i\eee^{-i\theta}) + \theta\big| = \big|\Im(\eee^{-i\theta}) + \theta\big| \lesssim |\theta|^3
\end{equation}
and, using \eqref{eq:explicit-2},
\begin{equation}
  \big|\eee^{i(\zeta - \theta)} + i\eee^{-i\theta}\big| = \big|\eee^{i\zeta} + i\big| \leq \big|\zeta + \frac{\pi}{2}\big|,
\end{equation}
hence
\begin{equation}
  \label{eq:real-part}
  \big|\Re(\eee^{i(\zeta - \theta)}) + \theta\big| \lesssim |\theta|^3 + \big|\zeta + \frac{\pi}{2}\big| \lesssim |t|^{-\frac{3}{N-6}}.
\end{equation}
The bound \eqref{eq:energy-expansion-2} follows now from \eqref{eq:real-part}, which finishes the proof.
\end{proof}
\begin{lemma}
  \label{lem:energy-linear}
  Under the assumptions of Proposition~\ref{prop:coercivity}, there holds
  \begin{equation}
    \label{eq:energy-linear}
    \big|\la \vD E(\eee^{i\zeta}W_\mu + \eee^{i\theta}W_\lambda), g\ra\big| \lesssim \|g\|_\cE\cdot \lambda^\frac{N+2}{4}.
  \end{equation}
\end{lemma}
\begin{proof}
  Using the fact that $\vD E(\eee^{i\zeta}W_\mu) = \vD E(\eee^{i\theta}W_\lambda) = 0$, \eqref{eq:energy-linear} is seen
  to be equivalent to
  \begin{equation}
    \label{eq:energy-linear-1}
    \big|\la f(\eee^{i\zeta} W_\mu + \eee^{i\theta}W_\lambda) - f(\eee^{i\zeta}W_\mu) - f(\eee^{i\theta}W_\lambda), g\ra\big| \lesssim \|g\|_\cE\cdot \lambda^\frac{N+2}{4}.
  \end{equation}
  By the Sobolev inequality, it suffices to check that
  \begin{equation}
    \label{eq:energy-linear-2}
    \|f(\eee^{i\zeta} W_\mu + \eee^{i\theta}W_\lambda) - f(\eee^{i\zeta} W_\mu) - f(\eee^{i\theta}W_\lambda)\|_{L^\frac{2N}{N+2}} \lesssim \lambda^\frac{N+2}{4}.
  \end{equation}
  As usual, we consider separately the regions $|x| \leq \sqrt\lambda$ and $|x| \geq \sqrt\lambda$. In the first region we have $W_\mu \lesssim W_\lambda$,
  hence \eqref{eq:pointwise-2} with $z_1 = W_\lambda$ and $z_2 = W_\mu$ yields
  \begin{equation}
    \|f(\eee^{i\zeta} W_\mu + \eee^{i\theta}W_\lambda) - f(\eee^{i\zeta} W_\mu) - f(\eee^{i\theta}W_\lambda)\|_{L^\frac{2N}{N+2}} \lesssim W_\lambda^\frac{4}{N-2}W_\mu + W_\mu^\frac{N+2}{N-2} \lesssim W_\lambda^\frac{4}{N-2}W_\mu \lesssim W_\lambda^\frac{4}{N-2}.
  \end{equation}
  By a change of variable we obtain
  \begin{equation}
    \begin{aligned}
      \|W_\lambda^\frac{4}{N-2}\|_{L^\frac{2N}{N+2}(|x| \leq \sqrt\lambda)} &= \lambda^{N\cdot \frac{N+2}{2N} - \frac{N-2}{2}\cdot \frac{4}{N-2}}\|W^\frac{4}{N-2}\|_{L^\frac{2N}{N+2}(|x| \leq 1/\sqrt\lambda)} \\
      &\lesssim \lambda^\frac{N-2}{2}\Big(\int_0^{1/\sqrt\lambda} r^{-4\frac{2N}{N+2}}r^{N-1}\ud r\Big)^\frac{N+2}{2N} \sim \lambda^{\frac{N-2}{2} - \frac{(N-6)N}{2(N+2)}\cdot \frac{N+2}{2N}}=\lambda^\frac{N+2}{4}.
    \end{aligned}
  \end{equation}

  In the region $|x| \geq \sqrt\lambda$ we have $W_\lambda \lesssim W_\mu$, hence \eqref{eq:pointwise-2} with $z_1 = W_\mu$ and $z_2 = W_\lambda$ yields
  \begin{equation}
    \|f(\eee^{i\zeta} W_\mu + \eee^{i\theta}W_\lambda) - f(\eee^{i\zeta} W_\mu) - f(\eee^{i\theta}W_\lambda)\|_{L^\frac{2N}{N+2}} \lesssim W_\mu^\frac{4}{N-2}W_\lambda + W_\lambda^\frac{N+2}{N-2} \lesssim W_\mu^\frac{4}{N-2}W_\lambda \lesssim W_\lambda,
  \end{equation}
and we have
\begin{equation}
  \begin{aligned}
    \|W_\lambda\|_{L^\frac{2N}{N+2}(|x|\geq \sqrt\lambda)} &= \lambda^2 \|W\|_{L^\frac{2N}{N+2}(|x| \geq 1/\sqrt\lambda)} \\
    &\lesssim \lambda^2 \Big(\int_{1/\sqrt\lambda}^{+\infty}r^{-(N-2)\cdot \frac{2N}{N+2}}r^{N-1}\ud r\Big)^\frac{N+2}{2N} \sim \lambda^{2+\frac{(N-6)N}{2(N+2)}\cdot \frac{N+2}{2N}} = \lambda^\frac{N+2}{4}.
  \end{aligned}
\end{equation}
\end{proof}

We now examine coercivity of the quadratic part in \eqref{eq:energy-taylor}.

\begin{lemma}
  \label{lem:coer-Lp-Lm}
  There exist constants $c, C > 0$ such that
  \begin{itemize}
    \item for any real-valued radial $g \in \cE$ there holds
      \begin{gather}
        \label{eq:coer-Lp-1}
        \la g, L^+g\ra \geq c\int_{\bR^N}|\grad g|^2 \ud x -C\big(\la W, g\ra^2 + \la \cY^{(2)}, g\ra^2\big), \\
        \label{eq:coer-Lm-1}
        \la g, L^-g\ra \geq c\int_{\bR^N}|\grad g|^2 \ud x -C\la \Lambda W, g\ra^2,
      \end{gather}
    \item if $r_1 > 0$ is large enough, then for any real-valued radial $g \in \cE$ there holds
      \begin{gather}
        \label{eq:coer-Lp-2}
        (1-2c)\int_{|x|\leq r_1}|\grad g|^2 \ud x + c\int_{|x|\geq r_1}|\grad g|^2\ud x + \int_{\bR^N}V^+|g|^2\ud x \geq -C\big(\la W, g\ra^2 + \la \cY^{(2)}, g\ra^2\big), \\
        \label{eq:coer-Lm-2}
        (1-2c)\int_{|x|\leq r_1}|\grad g|^2 \ud x + c\int_{|x|\geq r_1}|\grad g|^2\ud x + \int_{\bR^N}V^-|g|^2\ud x \geq -C\la \Lambda W, g\ra^2,
      \end{gather}
    \item if $r_2 > 0$ is small enough, then for any real-valued radial $g \in \cE$ there holds
      \begin{gather}
        \label{eq:coer-Lp-3}
        (1-2c)\int_{|x|\geq r_2}|\grad g|^2 \ud x + c\int_{|x|\leq r_2}|\grad g|^2\ud x + \int_{\bR^N}V^+|g|^2\ud x \geq -C\big(\la W, g\ra^2 + \la \cY^{(2)}, g\ra^2\big), \\
        \label{eq:coer-Lm-3}
        (1-2c)\int_{|x|\geq r_2}|\grad g|^2 \ud x + c\int_{|x|\leq r_2}|\grad g|^2\ud x + \int_{\bR^N}V^-|g|^2\ud x \geq -C\la \Lambda W, g\ra^2.
      \end{gather}
  \end{itemize}
\end{lemma}
\begin{proof}
  In the proofs of \eqref{eq:coer-Lp-1} and \eqref{eq:coer-Lm-1} we repeat with minor modifications the arguments of Nakanishi and Roy \cite{NaRo15p}. We include them for the reader's convenience.

  Let us show that
  \begin{equation}
    \label{eq:coer-Lp-11}
    g \in \cE, \la \cY^{(2)}, g\ra = 0\quad\Rightarrow\quad \la g, L^+, g\ra \geq 0.
  \end{equation}
  Suppose the contrary. Let $(a, b) \in \bR^2 \setminus\{(0, 0)\}$ and consider $ag + b\cY^{(1)} \in \cE$.
  Since $\cY^{(2)} \neq W$, \eqref{eq:Y1Y2} yields
  \begin{equation}
    \label{eq:Y1Y2pos}
    \la \cY^{(1)}, L^+ \cY^{(1)}\ra = -\nu\la \cY^{(1)}, \cY^{(2)} \ra = -\la L^-\cY^{(2)}, \cY^{(2)}\ra < 0,
  \end{equation}
  so we obtain
  \begin{equation}
    \begin{aligned}
      \la ag + bY^{(1)}, L^+(ag + b\cY^{(1)})\ra &= a^2 \la g, L^+ g\ra + 2ab \la g, L^+\cY^{(1)}\ra + b^2 \la \cY^{(1)}, L^+\cY^{(1)}\ra \\
      &= a^2 \la g, L^+ g\ra - 2ab\nu \la g, \cY^{(2)}\ra + b^2 \la \cY^{(1)}, L^+\cY^{(1)}\ra < 0.
    \end{aligned}
  \end{equation}
  This is impossible, because $L^+$ has only one negative direction. This proves \eqref{eq:coer-Lp-11}.

  Suppose \eqref{eq:coer-Lp-1} fails. Then there exists a sequence $g_n \in \cE$ such that $\|g_n\|_\cE = 1$
  and
  \begin{equation}
    \label{eq:coer-Lp-12}
    \la g_n, L^+ g_n\ra \leq c_n - C_n\big(\la W, g_n\ra^2 + \la \cY^{(2)}, g\ra^2\big),\qquad c_n \to 0,\quad C_n \to +\infty.
  \end{equation}
  Upon extracting a subsequence, we can assume that $g_n \wto g \in \cE$. Since $|\la g_n, L^+ g_n\ra| \lesssim \|g_n\|_\cE^2 = 1$,
  from \eqref{eq:coer-Lp-12} we immediately get $\la W, g\ra = \la \cY^{(2)}, g\ra = 0$.
  Also, by standard arguments $\la g_n, V^+ g_n\ra \to \la g, V^+ g\ra$, hence by the Fatou property
  \begin{equation}
    \la g, L^+ g\ra \leq \liminf_n \la g_n, L^+, g_n\ra \leq \liminf_n c_n = 0.
  \end{equation}
  Thus $g$ is a minimizer for the quadratic form associated with $L^+$ on the hyperplane orthogonal to $\cY^{(2)}$.
  This implies that $\la h, L^+ g\ra = 0$ for all $h \in \cE$ such that $\la \cY^{(2)}, h\ra = 0$.
  But we also have $\la \cY^{(1)}, L^+ g\ra = \la L^+ \cY^{(1)}, g\ra = -\nu\la \cY^{(2)}, g\ra = 0$
  and $\la \cY^{(1)}, \cY^{(2)}\ra \neq 0$, see \eqref{eq:Y1Y2pos}, so we obtain $\la h, L^+ g\ra = 0$ for all $h \in \cE$.
  Hence $g = \Lambda W$. But $\la W, \Lambda W\ra = -\|W\|_{L^2}^2 \neq 0$, so we get a contradiction. This proves \eqref{eq:coer-Lp-1}.

  The proof of \eqref{eq:coer-Lm-1} is similar. We obtain that the weak limit $g$ is a minimizer
  for the quadratic form associated with $L^-$ (without constraints), hence $g = W$,
  which is incompatible with the orthogonality condition.

  Once we have \eqref{eq:coer-Lp-1} \eqref{eq:coer-Lm-1},
  the bounds \eqref{eq:coer-Lp-2}, \eqref{eq:coer-Lm-2}, \eqref{eq:coer-Lp-3} and \eqref{eq:coer-Lm-3}
  follow by repeating the~proof of Lemma 2.1 in \cite{moi15p-3}.
\end{proof}
We now use this lemma to study the linearization around $\eee^{i\theta}W_\lambda$ for a complex-valued perturbation~$g$.
\begin{proposition}
  \label{prop:coer-L}
  There exist constants $c, C > 0$ such that for any $\theta \in \bR$ and $\lambda > 0$
  \begin{itemize}
    \item for any complex-valued radial $g \in \cE$ there holds
      \begin{equation}
        \label{eq:coer-L-1}
        \begin{gathered}
        \int_{\bR^N}|\grad g|^2 \ud x - \Re \int_{\bR^N}\conj g\cdot f'(\eee^{i\theta}W_\lambda)g\ud x \geq  \\
        \geq c\int_{\bR^N}|\grad g|^2 \ud x -C\big(\la \lambda^{-2}\eee^{i\theta}W_\lambda, g\ra^2 + \la \lambda^{-2}i\eee^{i\theta}\Lambda W_\lambda, g\ra^2 + \la \alpha_{\theta, \lambda}^+, g\ra^2 + \la \alpha_{\theta, \lambda}^-, g\ra^2\big),
      \end{gathered}
      \end{equation}
    \item if $r_1 > 0$ is large enough, then for any complex-valued radial $g \in \cE$ there holds
      \begin{equation}
        \label{eq:coer-L-2}
        \begin{gathered}
        (1-2c)\int_{|x|\leq r_1}|\grad g|^2 \ud x + c\int_{|x|\geq r_1}|\grad g|^2\ud x - \Re \int_{\bR^N}\conj g\cdot f'(\eee^{i\theta}W_\lambda)g\ud x \geq \\
        \geq {-}C\big(\la \lambda^{-2}\eee^{i\theta}W_\lambda, g\ra^2 + \la \lambda^{-2}i\eee^{i\theta}\Lambda W_\lambda, g\ra^2 + \la \alpha_{\theta, \lambda}^+, g\ra^2 + \la \alpha_{\theta, \lambda}^-, g\ra^2\big),
      \end{gathered}
      \end{equation}
    \item if $r_2 > 0$ is small enough, then for any complex-valued radial $g \in \cE$ there holds
      \begin{equation}
        \label{eq:coer-L-3}
        \begin{gathered}
          (1-2c)\int_{|x|\geq r_2}|\grad g|^2 \ud x + c\int_{|x|\leq r_2}|\grad g|^2\ud x - \Re \int_{\bR^N}\conj g\cdot f'(\eee^{i\theta}W_\lambda)g\ud x \geq \\
        \geq {-}C\big(\la \lambda^{-2}\eee^{i\theta}W_\lambda, g\ra^2 + \la \lambda^{-2}i\eee^{i\theta}\Lambda W_\lambda, g\ra^2 + \la \alpha_{\theta, \lambda}^+, g\ra^2 + \la \alpha_{\theta, \lambda}^-, g\ra^2\big),
      \end{gathered}
      \end{equation}
  \end{itemize}
\end{proposition}
\begin{remark}
  Note that the scalar products on the right hand side of these estimates are the ones which appear in the orthogonality conditions in the previous section. For the definition of $\alpha_{\theta, \lambda}^\pm$, see \eqref{eq:alpha}.
\end{remark}
\begin{proof}
  Without loss of generality we can assume that $\theta = 0$ and $\lambda = 1$.
  Let $g = g_1 + i g_2$. Observe that
  \begin{equation}
    -f'(W)(g_1 + i g_2) = -W^\frac{4}{N-2}(g_1 + i g_2) - \frac{4}{N-2}W^\frac{4}{N-2}g_1 = V^+ g_1 + iV^- g_2,
  \end{equation}
  which gives
  \begin{equation}
    -\Re\int_{\bR^N}\conj g\cdot f'(W)g\ud x = \int_{\bR^N} V^+ g_1^2 \ud x + \int_{\bR^N} V^- g_2^2\ud x.
  \end{equation}
  Also, $\la W, g\ra = \la W, g_1\ra$ and $\la i\Lambda W, g\ra = \la \Lambda W, g_2\ra$.
  We have $\cY^{(2)} = \frac 12(\alpha_{0, 1}^+ + \alpha_{0, 1}^-)$, so
  \begin{equation}
    \la \cY^{(2)}, g_1\ra^2 = \la \cY^{(2)}, g\ra^2 \leq \frac 12 \big(\la \alpha_{0, 1}^+, g\ra^2 + \la \alpha_{0, 1}^-, g\ra^2\big).
  \end{equation}
  Applying \eqref{eq:coer-Lp-1} with $g = g_1$ and \eqref{eq:coer-Lm-1} with $g = g_2$ we obtain \eqref{eq:coer-L-1}.
  The proofs of \eqref{eq:coer-L-2} and \eqref{eq:coer-L-3} are similar.
\end{proof}
One consequence of the last proposition is the coercivity near a sum of two bubbles at different scales:
\begin{lemma}
  \label{lem:coer-L-two}
  There exist $\eta, C > 0$ such that if $\lambda \leq \eta\mu$, then for all $g \in \cE$ satisfying \eqref{eq:orth}
  there holds
  \begin{equation}
    \label{eq:coer-L-two}
    \frac 1C \|g\|_\cE^2 \leq \frac 12 \la \vD^2 E(\eee^{i\zeta}W_\mu + \eee^{i\theta}W_\lambda)g, g\ra + 2\big((a_1^+)^2 + (a_1^-)^2 + (a_2^+)^2 + (a_2^-)^2\big) \leq C\|g\|_\cE^2
  \end{equation}
\end{lemma}
\begin{proof}
  It is essantially the same as the proof of \cite[Lemma 3.5]{moi15p-3}.
\end{proof}
\begin{proof}[Proof of Proposition~\ref{prop:coercivity}]
  Bound \eqref{eq:coer-bound} follows immediately from \eqref{eq:energy-taylor}, Lemmas~\ref{lem:coer-sans-g}, \ref{lem:energy-linear},
  \ref{lem:coer-L-two} and the triangle inequality.

  For any $c > 0$ we have $\|g\|_\cE^\frac{2N}{N-2} \leq c\|g\|_\cE^2$ if $\eta$ is chosen small enough,
  hence \eqref{eq:energy-taylor} and Lemmas~\ref{lem:coer-sans-g}, \ref{lem:energy-linear} yield
  \begin{equation}
    \begin{aligned}
    &\Big|E(u) - 2E(W) - \frac 1N\big(N(N-2)\big)^\frac N2\theta\lambda^\frac{N-2}{2} - \frac 12\la \vD^2 E(\eee^{i\zeta}W_\mu + \eee^{i\theta}W_\lambda)g, g\ra\Big| \\
    &\leq C\big(\big|\zeta + \frac{\pi}{2}\big| + |\mu - 1| + |\theta|^3 + \lambda\big)\lambda^\frac{N-2}{2} + c\|g\|_\cE^2,
  \end{aligned}
  \end{equation}
  hence
  \begin{equation}
    \begin{aligned}
    &\frac 1N\big(N(N-2)\big)^\frac N2\theta\lambda^\frac{N-2}{2} + \frac 12\la \vD^2 E(\eee^{i\zeta}W_\mu + \eee^{i\theta}W_\lambda)g, g\ra \\
    &\leq E(u) - 2E(W) + C\big(\big|\zeta + \frac{\pi}{2}\big| + |\mu - 1| + |\theta|^3 + \lambda\big)\lambda^\frac{N-2}{2} + c\|g\|_\cE^2.
  \end{aligned}
  \end{equation}
  Choosing $c$ small enough and invoking Lemma~\ref{lem:coer-L-two} finishes the proof of \eqref{eq:coer-conclusion}.
\end{proof}
\section{Modulation}
\label{sec:mod}
\subsection{Bounds on the modulation parameters}
We study solutions of the following form:
\begin{equation}
  \label{eq:decompose}
  u(t) = \eee^{i\zeta(t)}W_{\mu(t)} + \eee^{i\theta(t)}W_{\lambda(t)} + g(t),
\end{equation}
with
\begin{equation}
  \label{eq:param-rough}
  |\mu(t) - 1| \ll 1,\quad \big|\zeta(t)+\frac{\pi}{2}\big| \ll 1,\quad \lambda(t) \ll 1,\quad |\theta(t)| \ll 1\quad\text{and}\quad \|g\|_\cE \ll 1.
\end{equation}
We will often omit the time variable and write $\zeta$ for $\zeta(t)$ etc.

Differentiating \eqref{eq:decompose} in time we obtain
  \begin{equation}
    \label{eq:dtu}
    \partial_t u = \zeta'i\eee^{i\zeta}W_\mu - \frac{\mu'}{\mu}\eee^{i\zeta}\Lambda W_\mu + \theta'i\eee^{i\theta}W_\lambda - \frac{\lambda'}{\lambda}\Lambda W_\lambda + \partial_t g.
  \end{equation}
  On the other hand, using $\Delta(W_\mu) + f(W_\mu) = \Delta(W_\lambda) + f(W_\lambda) = 0$ we get
  \begin{equation}
    \label{eq:rhsu}
    i\Delta u + if(u) = i\Delta g + i\big(f(\eee^{i\zeta}W_\mu + \eee^{i\theta}W_\lambda + g) - f(\eee^{i\zeta}W_\mu) - f(\eee^{i\theta}W_\lambda)\big),
  \end{equation}
  hence \eqref{eq:nls} yields
  \begin{equation}
    \label{eq:dtg}
    \begin{aligned}
    \partial_t g &= i\Delta g + i\big(f(\eee^{i\zeta}W_{\mu} + \eee^{i\theta}W_{\lambda} + g) - f(\eee^{i\zeta}W_\mu) - f(\eee^{i\theta}W_\lambda)\big)  \\
    &-\zeta' i\eee^{i\zeta}W_\mu + \frac{\mu'}{\mu}\eee^{i\zeta}\Lambda W_\mu - \theta' i\eee^{i\theta}W_\lambda + \frac{\lambda'}{\lambda}\eee^{i\theta}\Lambda W_\lambda.
    \end{aligned}
  \end{equation}
  Since we work with non-classical solutions, it is worth pointing out
  that the equation above should be understood as a notational simplification.
  Any computation involving $g(t)$ could be rewritten in terms of $u(t)$
  and the modulation parameters $\zeta$, $\mu$, $\theta$, $\lambda$.
  Most of the time we only use the fact that \eqref{eq:dtg} holds in the weak sense,
  but later we will also need to compute the time derivative of a quadratic form in $g(t)$,
  in which case the rigourous meaning of the computation is less clear.

  We impose the orthogonality conditions \eqref{eq:orth}. By standards arguments using
  the Implicit Function theorem, they uniquely determine the modulation parameters.

We need precise bootstrap assumptions about the parameters quantifying \eqref{eq:param-rough}.
In order to formulate them, denote
\begin{equation}
  \label{eq:kappa}
  \kappa := \Big(\frac{N-6}{N\cdot B\big(\frac{N-4}{2}, \frac{N}{2}\big)}\Big)^{\frac{2}{N-4}}.
\end{equation}
\begin{lemma}
  \label{lem:mod}
  Let $c > 0$ be an arbitrarily small constant. Let $T_0 < 0$ with $|T_0|$ large enough (depending on $c$)
  and $T < T_1 \leq T_0$. Suppose that for $T \leq t \leq T_1$ there holds
  \begin{align}
  \big|\zeta(t) + \frac{\pi}{2}\big| &\leq |t|^{-\frac{3}{N-6}}, \label{eq:bootstrap-zeta} \\
  |\mu(t) - 1| &\leq |t|^{-\frac{3}{N-6}}, \label{eq:bootstrap-mu} \\
  |\theta(t)| &\leq |t|^{-\frac{1}{N-6}}, \label{eq:bootstrap-theta} \\
  \big|\lambda(t) - \frac{1}{\kappa}(\kappa|t|)^{-\frac{2}{N-6}}\big| &\leq |t|^{-\frac{5}{2(N-6)}}, \label{eq:bootstrap-lambda} \\
  \|g\|_\cE &\leq |t|^{-\frac{N-1}{2(N-6)}}. \label{eq:bootstrap-g}
\end{align}
  Then
  \begin{align}
    \label{eq:mod-zeta}
    |\zeta'(t)| &\leq c|t|^{-\frac{N-3}{N-6}}, \\
    \label{eq:mod-mu}
    |\mu'(t)| &\leq c|t|^{-\frac{N-3}{N-6}}, \\
    \label{eq:mod-l}
    \Big|\lambda'(t) - \frac{2\kappa^\frac{N-4}{2}}{N-6}\lambda(t)^\frac{N-4}{2}\Big| &\leq c|t|^{-\frac{2N-7}{2(N-6)}}, \\
    \label{eq:mod-th}
    \Big|\theta'(t)+\frac{(N-2)\kappa^\frac{N-4}{2}}{N-6}\theta(t)\lambda(t)^\frac{N-6}{2}-\frac{K(t)}{\lambda(t)^2\|W\|_{L^2}^2}\Big| &\leq c|t|^{-\frac{N-5}{N-6}},
  \end{align}
  for $T \leq t \leq T_1$, where
  \begin{equation}
    \label{eq:K-def}
    K := -\big\la \eee^{i\theta}\Lambda W_\lambda, f(\eee^{i\zeta}W_{\mu} + \eee^{i\theta}W_{\lambda} + g) - f(\eee^{i\zeta}W_\mu + \eee^{i\theta}W_\lambda) - f'(\eee^{i\zeta}W_\mu + \eee^{i\theta}W_\lambda)g \big\ra.
\end{equation}
\end{lemma}
\begin{remark}
  \label{rem:bootstrap-lambda}
  We will not really use \eqref{eq:bootstrap-lambda}, but only the fact that $\lambda(t) \sim |t|^{-\frac{2}{N-6}}$.
\end{remark}
\begin{proof}
  We use the usual method of differentiating the orthogonality conditions in time, which will yield a linear system of the form:
  \begin{equation}
    \label{eq:mod-system}
    \begin{pmatrix}
      M_{11} & M_{12} & M_{13} & M_{14} \\ M_{21} & M_{22} & M_{23} & M_{24} \\ M_{31} & M_{32} & M_{33} & M_{34} \\ M_{41} & M_{42} & M_{43} & M_{44}
    \end{pmatrix} \begin{pmatrix}\mu^2 \zeta' \\ \mu \mu' \\ \lambda^2 \theta' \\ \lambda\lambda'\end{pmatrix} = \begin{pmatrix}B_1 \\ B_2 \\ B_3 \\ B_4 \end{pmatrix}.
  \end{equation}
  Here, the coefficients $M_{ij}$ and $B_i$ depend on $g$, $\zeta$, $\mu$, $\theta$ and $\lambda$. We will now compute all these coefficients and prove appropriate bounds.

  \textbf{First row.}
  Differentiating $\la i\eee^{i\zeta}\Lambda W_\mu, g\ra = 0$ and using \eqref{eq:dtg} we obtain
  \begin{equation}
    \begin{aligned}
      0 &= \dd t \la i\eee^{i\zeta}\Lambda W_\mu, g\ra = -\zeta'\la \eee^{i\zeta}\Lambda W_\mu, g\ra - \frac{\mu'}{\mu}\la i\eee^{i\zeta}\Lambda\Lambda W_\mu, g\ra + \la i\eee^{i\zeta}\Lambda W_\mu, \partial_t g\ra \\
      &= \zeta'\big({-}\la i\eee^{i\zeta}\Lambda W_\mu, i\eee^{i\zeta} W_\mu\ra - \la \eee^{i\zeta}\Lambda W_\mu, g\ra\big) + \frac{\mu'}{\mu}\big(\la i\eee^{i\zeta}\Lambda W_\mu, \eee^{i\zeta}\Lambda W_\mu\ra - \la i\eee^{i\zeta}\Lambda\Lambda W_\mu, g\ra\big) \\
      &+ \theta'\la i\eee^{i\zeta}\Lambda W_\mu, -i\eee^{i\theta} W_\lambda\ra + \frac{\lambda'}{\lambda}\la i\eee^{i\zeta}\Lambda W_\mu, \eee^{i\theta}\Lambda W_\lambda\ra \\
      &+ \big\la i\eee^{i\zeta}\Lambda W_\mu, i\Delta g + i\big(f(\eee^{i\zeta}W_{\mu} + \eee^{i\theta}W_{\lambda} + g) - f(\eee^{i\zeta}W_\mu) - f(\eee^{i\theta}W_\lambda)\big) \big\ra.
  \end{aligned}
  \end{equation}
  Note that $\la -\Lambda W_\mu, W_\mu\ra = \|W_\mu\|_{L^2}^2 = \mu^2 \|W\|_{L^2}^2$, hence we get
  \begin{align}
    M_{11} &= \mu^{-2}\big({-}\la i\eee^{i\zeta}\Lambda W_\mu, i\eee^{i\zeta} W_\mu\ra - \la \eee^{i\zeta}\Lambda W_\mu, g\ra\big) = \|W\|_{L^2}^2 + O(\|g\|_\cE) = \|W\|_{L^2}^2 + O(|t|^{-\frac{N-1}{2(N-6)}}), \\
    M_{12} &= \mu^{-2}\big(\la i\eee^{i\zeta}\Lambda W_\mu, \eee^{i\zeta}\Lambda W_\mu\ra - \la i\eee^{i\zeta}\Lambda\Lambda W_\mu, g\ra\big) = O(\|g\|_\cE) = O(|t|^{-\frac{N-1}{2(N-6)}}), \\
    M_{13} &= \lambda^{-2}\la i\eee^{i\zeta}\Lambda W_\mu, -i\eee^{i\theta} W_\lambda\ra = O(1), \\
    M_{14} &= \lambda^{-2}\la i\eee^{i\zeta}\Lambda W_\mu, \eee^{i\theta}\Lambda W_\lambda\ra = O(1).
  \end{align}

  Let us consider the term
  \begin{equation}
    \label{eq:B1}
    B_1 = -\big\la i\eee^{i\zeta}\Lambda W_\mu, i\Delta g + i\big(f(\eee^{i\zeta}W_{\mu} + \eee^{i\theta}W_{\lambda} + g) - f(\eee^{i\zeta}W_\mu) - f(\eee^{i\theta}W_\lambda)\big) \big\ra.
  \end{equation}
  From \eqref{eq:conj-ker-2} (with $\theta$ replaced by $\zeta$ and $\lambda$ replaced by $\mu$) we obtain
  \begin{equation}
    \begin{aligned}
    B_1 &= -\mu^{-2}\big\la i\eee^{i\zeta}\Lambda W_\mu, i\big(f(\eee^{i\zeta}W_{\mu} + \eee^{i\theta}W_\lambda + g) - f(\eee^{i\zeta}W_\mu) - f(\eee^{i\theta}W_\lambda) - f'(\eee^{i\zeta}W_\mu)g\big) \big\ra \\
    &= -\mu^{-2}\big\la \eee^{i\zeta}\Lambda W_\mu, \big(f(\eee^{i\zeta}W_{\mu} + \eee^{i\theta}W_\lambda + g) - f(\eee^{i\zeta}W_\mu) - f(\eee^{i\theta}W_\lambda) - f'(\eee^{i\zeta}W_\mu)g\big) \big\ra.
  \end{aligned}
  \end{equation}
  First we show that
  \begin{equation}
    \label{eq:B1-estim-1}
    \big|\big\la \eee^{i\zeta}\Lambda W_\mu, f(\eee^{i\zeta}W_\mu + \eee^{i\theta}W_\lambda + g) - f(\eee^{i\zeta}W_\mu + \eee^{i\theta}W_\lambda) - f'(\eee^{i\zeta}W_\mu + \eee^{i\theta}W_\lambda)g\big\ra\big| \lesssim \|g\|_\cE^2.
  \end{equation}
  Note that \eqref{eq:bootstrap-zeta} and \eqref{eq:bootstrap-theta} imply that $|\eee^{i\zeta}W_\mu + \eee^{i\theta}W_\lambda| \gtrsim W_\mu$, hence
  \eqref{eq:pointwise-1} with $z_1 = \eee^{i\zeta}W_\mu + \eee^{i\theta}W_\lambda$ and $z_2 = g$ yields
  \begin{equation}
    \label{eq:B1-estim-11}
    |f(\eee^{i\zeta}W_\mu + \eee^{i\theta}W_\lambda + g) - f(\eee^{i\zeta}W_\mu + \eee^{i\theta}W_\lambda) - f'(\eee^{i\zeta}W_\mu + \eee^{i\theta}W_\lambda)g| \lesssim W_\mu^{-\frac{N-6}{N-2}}|g|^2.
  \end{equation}
  Using the fact that $|\Lambda W| \lesssim W$ and the H\"older inequality we arrive at \eqref{eq:B1-estim-1}.

  Next we show that
  \begin{equation}
    \label{eq:B1-estim-2}
    \big|\big\la \eee^{i\zeta}\Lambda W_\mu, f(\eee^{i\zeta}W_\mu + \eee^{i\theta}W_\lambda) -
    f(\eee^{i\zeta}W_\mu) - f(\eee^{i\theta}W_\lambda)\big\ra\big| \lesssim \lambda^\frac{N-2}{2}.
  \end{equation}
  Using \eqref{eq:pointwise-2} we get
  \begin{equation}
    \label{eq:B1-estim-21}
    |f(\eee^{i\zeta}W_\mu + \eee^{i\theta}W_\lambda) -
    f(\eee^{i\zeta}W_\mu) - f(\eee^{i\theta}W_\lambda)| \lesssim |W_\mu|^\frac{4}{N-2}W_\lambda + |f(W_\lambda)|.
  \end{equation}
  The second term is easy. We have $f(W) \in L^1$ and we check that $\|f(W_\lambda)\|_{L^1} \sim \lambda^\frac{N-2}{2}$ by a change of variable.
  Consider the first term. In the region $|x| \leq 1$ we write
  \begin{equation}
    \|W_\lambda\|_{L^1(|x| \leq 1)} = \lambda^\frac{N+2}{2}\|W\|_{L^1(|x| \leq \lambda^{-1})}
    \lesssim \lambda^\frac{N+2}{2}\int_0^{\lambda^{-1}}r^{-N+2}r^{N-1}\ud r \sim \lambda^\frac{N-2}{2}.
  \end{equation}
  As for $|x| \geq 1$, we notice that $\|W_\lambda\|_{L^\infty(|x| \geq 1)} \lesssim \lambda^\frac{N-2}{2}$
  and $|\Lambda W_\mu|\cdot|W_\mu|^\frac{4}{N-2}$ is bounded in $L^1$.

  Finally, we show that
  \begin{equation}
    \label{eq:B1-estim-3}
    \big|\big\la \eee^{i\zeta}\Lambda W_\mu, \big(f'(\eee^{i\zeta}W_\mu + \eee^{i\theta}W_\lambda) - f'(\eee^{i\zeta}W_\mu)\big)g\big\ra\big| \lesssim \lambda^\frac{N-2}{4}\|g\|_\cE.
  \end{equation}
In the region $|x| \leq \sqrt\lambda$ it suffices to use the bound
\begin{equation}
  |f'(\eee^{i\zeta}W_\mu + \eee^{i\theta}W_\lambda) - f'(\eee^{i\zeta}W_\mu)| \lesssim W_\lambda^\frac{4}{N-2}
\end{equation}
and the fact that
\begin{equation}
  \label{eq:B1-estim-31}
  \|W_\lambda^\frac{4}{N-2}\|_{L^\frac{2N}{N+2}(|x|\leq \sqrt\lambda)} = \lambda^\frac{N-2}{2}\|W^\frac{4}{N-2}\|_{L^\frac{2N}{N+2}(|x| \leq \lambda^{-\frac 12})} \lesssim \lambda^\frac{N+2}{4},
\end{equation}
where the last inequality follows from $W^\frac{4}{N-2}(x) \lesssim |x|^{-4}$.
In the region $|x| \geq \sqrt\lambda$ we use H\"older and the fact that
\begin{equation}
  \label{eq:B1-estim-32}
  \|W_\lambda\|_{L^\frac{2N}{N-2}(|x| \geq \sqrt\lambda)} =\|W\|_{L^\frac{2N}{N-2}(|x|\geq \lambda^{-\frac 12})}
  \lesssim \Big(\int_{\lambda^{-\frac 12}}^{+\infty}r^{-2N}r^{N-1}\ud r\Big)^\frac{N-2}{2N} \lesssim \lambda^\frac{N-2}{4}.
\end{equation}

Taking the sum of \eqref{eq:B1-estim-1}, \eqref{eq:B1-estim-2}, \eqref{eq:B1-estim-3} and using \eqref{eq:bootstrap-lambda}, \eqref{eq:bootstrap-g} we obtain
  \begin{equation}
    \label{eq:B1-estim}
    |B_1| \lesssim |t|^{-\frac{N-2}{N-6}}.
  \end{equation}

  \textbf{Second row.}
  Differentiating $\la -\eee^{i\zeta}W_\mu, g\ra = 0$ we obtain
  \begin{equation}
    \begin{aligned}
      0 &= \dd t \la -\eee^{i\zeta}W_\mu, g\ra = -\zeta'\la i\eee^{i\zeta}W_\mu, g\ra + \frac{\mu'}{\mu}\la \eee^{i\zeta}\Lambda W_\mu, g\ra - \la \eee^{i\zeta}W_\mu, \partial_t g\ra \\
      &= \zeta'\big(\la \eee^{i\zeta}W_\mu, i\eee^{i\zeta} W_\mu\ra - \la i\eee^{i\zeta}\Lambda W_\mu, g\ra\big) + \frac{\mu'}{\mu}\big({-}\la \eee^{i\zeta}W_\mu, \eee^{i\zeta}\Lambda W_\mu\ra + \la \eee^{i\zeta}\Lambda W_\mu, g\ra\big) \\
      &+ \theta'\la \eee^{i\zeta}W_\mu, i\eee^{i\theta} W_\lambda\ra + \frac{\lambda'}{\lambda}\la {-}\eee^{i\zeta}W_\mu, \eee^{i\theta}\Lambda W_\lambda\ra \\
      &- \big\la \eee^{i\zeta} W_\mu, i\Delta g + i\big(f(\eee^{i\zeta}W_{\mu} + \eee^{i\theta}W_{\lambda} + g) - f(\eee^{i\zeta}W_\mu) - f(\eee^{i\theta}W_\lambda)\big) \big\ra,
  \end{aligned}
  \end{equation}
  which yields
  \begin{align}
    M_{21} &= \mu^{-2}\big(\la \eee^{i\zeta} W_\mu, i\eee^{i\zeta} W_\mu\ra - \la i\eee^{i\zeta} W_\mu, g\ra\big) = O(\|g\|_\cE), \\
    M_{22} &= \mu^{-2}\big({-}\la \eee^{i\zeta}W_\mu, \eee^{i\zeta}\Lambda W_\mu\ra + \la \eee^{i\zeta}\Lambda W_\mu, g\ra\big) = \|W\|_{L^2}^2 + O(\|g\|_\cE), \\
    M_{23} &= \lambda^{-2}\la \eee^{i\zeta} W_\mu, i\eee^{i\theta} W_\lambda\ra = O(1), \\
    M_{24} &= \lambda^{-2}\la -\eee^{i\zeta} W_\mu, \eee^{i\theta}\Lambda W_\lambda\ra = O(1).
  \end{align}

  Consider now the term
  \begin{equation}
    \label{eq:B2}
    \begin{aligned}
    B_2 &= \big\la \eee^{i\zeta} W_\mu, i\Delta g + i\big(f(\eee^{i\zeta}W_{\mu} + \eee^{i\theta}W_{\lambda} + g) - f(\eee^{i\zeta}W_\mu) - f(\eee^{i\theta}W_\lambda)\big) \big\ra \\
    &= \big\la \eee^{i\zeta}W_\mu, i\big(f(\eee^{i\zeta}W_{\mu} + \eee^{i\theta}W_{\lambda} + g) - f(\eee^{i\zeta}W_\mu) - f(\eee^{i\theta}W_\lambda) - f'(\eee^{i\zeta}W_\mu)g\big) \big\ra, \\
  \end{aligned}
  \end{equation}
  where the second equality follows from \eqref{eq:conj-ker-1}.
  The proof of \eqref{eq:B1-estim} yields
\begin{equation}
  \label{eq:B2-estim}
  |B_2| \lesssim |t|^{-\frac{N-2}{N-6}}.
\end{equation}

  \textbf{Third row.}
  Differentiating $\la i\eee^{i\theta}\Lambda W_\lambda, g\ra = 0$ we obtain
  \begin{equation}
    \begin{aligned}
      0 &= \dd t \la i\eee^{i\theta}\Lambda W_\lambda, g\ra = -\theta'\la \eee^{i\theta}\Lambda W_\lambda, g\ra - \frac{\lambda'}{\lambda}\la i\eee^{i\theta}\Lambda\Lambda W_\lambda, g\ra + \la i\eee^{i\theta}\Lambda W_\lambda, \partial_t g\ra \\
      &= \zeta'\la i\eee^{i\theta}\Lambda W_\lambda, -i\eee^{i\zeta} W_\mu\ra + \frac{\mu'}{\mu}\la i\eee^{i\theta}\Lambda W_\lambda, \eee^{i\zeta}\Lambda W_\mu\ra \\
      &+ \theta'\big(\la i\eee^{i\theta}\Lambda W_\lambda, {-}i\eee^{i\theta} W_\lambda\ra -\la \eee^{i\theta}\Lambda W_\lambda, g\ra\big) + \frac{\lambda'}{\lambda}\big(\la i\eee^{i\theta}\Lambda W_\lambda, \eee^{i\theta}\Lambda W_\lambda\ra - \la i\eee^{i\theta}\Lambda\Lambda W_\lambda, g\ra\big) \\
      &+ \big\la i\eee^{i\theta}\Lambda W_\lambda, i\Delta g + i\big(f(\eee^{i\zeta}W_{\mu} + \eee^{i\theta}W_{\lambda} + g) - f(\eee^{i\zeta}W_\mu) - f(\eee^{i\theta}W_\lambda)\big) \big\ra,
  \end{aligned}
  \end{equation}
  which yields
  \begin{align}
    M_{31} &= \mu^{-2}\la i\eee^{i\theta}\Lambda W_\lambda, -i\eee^{i\zeta} W_\mu\ra =O(\lambda^2) =  O(|t|^{-\frac{4}{N-6}}), \\
    M_{32} &= \mu^{-2}\la i\eee^{i\theta}\Lambda W_\lambda, \eee^{i\zeta}\Lambda W_\mu\ra = O(\lambda^2) = O(|t|^{-\frac{4}{N-6}}), \\
    M_{33} &= \lambda^{-2}\big(\la i\eee^{i\theta}\Lambda W_\lambda, {-}i\eee^{i\theta} W_\lambda\ra -\la \eee^{i\theta}\Lambda W_\lambda, g\ra\big) = \|W\|_{L^2}^2 + O(\|g\|_\cE) = \|W\|_{L^2}^2 + O(|t|^{-\frac{N-1}{2(N-6)}}), \\
    M_{34} &= \lambda^{-2}\big(\la i\eee^{i\theta}\Lambda W_\lambda, \eee^{i\theta}\Lambda W_\lambda\ra - \la i\eee^{i\theta}\Lambda\Lambda W_\lambda, g\ra\big) = O(\|g\|_\cE) = O(|t|^{-\frac{N-1}{2(N-6)}}).
  \end{align}

  Let us consider the term
  \begin{equation}
    \begin{aligned}
    B_3 &= -\big\la i\eee^{i\theta}\Lambda W_\lambda, i\Delta g + i\big(f(\eee^{i\zeta}W_{\mu} + \eee^{i\theta}W_{\lambda} + g) - f(\eee^{i\zeta}W_\mu) - f(\eee^{i\theta}W_\lambda)\big) \big\ra \\
    &= -\big\la i\eee^{i\theta}\Lambda W_\lambda, i\big(f(\eee^{i\zeta}W_{\mu} + \eee^{i\theta}W_{\lambda} + g) - f(\eee^{i\zeta}W_\mu) - f(\eee^{i\theta}W_\lambda) - f'(\eee^{i\theta}W_\lambda)g\big) \big\ra \\
    &= -\big\la \eee^{i\theta}\Lambda W_\lambda, f(\eee^{i\zeta}W_{\mu} + \eee^{i\theta}W_{\lambda} + g) - f(\eee^{i\zeta}W_\mu) - f(\eee^{i\theta}W_\lambda) - f'(\eee^{i\theta}W_\lambda)g \big\ra,
  \end{aligned}
  \end{equation}
  where the second equality follows from \eqref{eq:conj-ker-2}.
  Comparing this formula with \eqref{eq:K-def} we obtain
  \begin{equation}
    \label{eq:B3-K}
    \begin{aligned}
      B_3 - K &= -\la \eee^{i\theta}\Lambda W_\lambda, f(\eee^{i\zeta}W_\mu + \eee^{i\theta}W_\lambda) - f(\eee^{i\zeta}W_\mu) - f(\eee^{i\theta}W_\lambda)\ra \\
      &- \big\la \eee^{i\theta} \Lambda W_\lambda, \big(f'(\eee^{i\zeta}W_\mu + \eee^{i\theta}W_\lambda) - f'(\eee^{i\theta}W_\lambda)\big)g\big\ra.
  \end{aligned}
  \end{equation}
  First we treat the second line by showing that
  \begin{equation}
    \label{eq:B3-estim-2}
    \big|\big\la \eee^{i\theta}\Lambda W_\lambda, \big(f'(\eee^{i\zeta}W_\mu + \eee^{i\theta}W_\lambda) - f'(\eee^{i\theta}W_\lambda)\big)g\big\ra\big| \lesssim \lambda^\frac{N}{4}\|g\|_\cE.
  \end{equation}
  We consider separately $|x| \leq \lambda^\gamma$ and $|x| \geq \lambda^\gamma$ with $\gamma = \frac{N-4}{2(N-2)}$.
  In the region $|x| \leq \lambda^\gamma$ we use the bound
  \begin{equation}
  |f'(\eee^{i\zeta}W_\mu + \eee^{i\theta}W_\lambda) - f'(\eee^{i\theta}W_\lambda)| \lesssim W_\lambda^{-\frac{N-6}{N-2}}W_\mu.
\end{equation}
It implies that
\begin{equation}
  |\eee^{i\theta}\Lambda W_\lambda| \cdot \big|\big(f'(\eee^{i\zeta}W_\mu + \eee^{i\theta}W_\lambda) - f'(\eee^{i\theta}W_\lambda)\big)g\big| \lesssim W_\lambda^\frac{4}{N-2}|g|
\end{equation}
pointwise and it suffices to see that
\begin{equation}
  \begin{aligned}
    \|W_\lambda^\frac{4}{N-2}\|_{L^{\frac{2N}{N+2}}(|x|\leq\lambda^\gamma)} &= \lambda^\frac{N-2}{2}\|W^\frac{4}{N-2}\|_{L^\frac{2N}{N+2}(|x|\leq \lambda^{\gamma-1})} \lesssim \lambda^\frac{N-2}{2}\Big(\int_0^{\lambda^{\gamma-1}}r^{-4\frac{2N}{N+2}}r^{N-1}\ud r\Big)^\frac{N+2}{2N} \\
    &\lesssim \lambda^\frac{N-2}{2}\lambda^{(\gamma-1)\frac{N(N-6)}{N+2}\cdot \frac{N+2}{2N}} = \lambda^{\frac{N-2}{2} - \frac{N(N-6)}{4(N-2)}} = \lambda^{\frac{N^2 - 2N + 8}{4(N-2)}} \ll \lambda^\frac N4.
  \end{aligned}
\end{equation}
In the region $|x| \geq \lambda^\gamma$ we have
\begin{equation}
  \begin{aligned}
  \|\Lambda W_\lambda\|_{L^{\frac{2N}{N-2}}(|x| \geq \lambda^\gamma)} &\lesssim \|W_\lambda\|_{L^{\frac{2N}{N-2}}(|x| \geq \lambda^\gamma)} = \|W\|_{L^{\frac{2N}{N-2}}(|x| \geq \lambda^{\gamma-1})} \\
  &\lesssim \Big(\int_{\lambda^{\gamma-1}}^{+\infty}r^{-2N}r^{N-1}\ud r\Big)^\frac{N-2}{2N} \lesssim \lambda^{(1-\gamma)N\frac{N-2}{2N}} = \lambda^\frac N4,
\end{aligned}
\end{equation}
which yields the required bound by H\"older.

We are left with the first line in \eqref{eq:B3-K}. We will prove that
\begin{equation}
  \label{eq:B3-estim-3}
  \Big|\la \eee^{i\theta}\Lambda W_\lambda, f(\eee^{i\zeta}W_\mu + \eee^{i\theta}W_\lambda) - f(\eee^{i\zeta}W_\mu) - f(\eee^{i\theta}W_\lambda)\ra - \frac{(N-2)\kappa^\frac{N-4}{2}\|W\|_{L^2}^2}{N-6}\theta\lambda^\frac{N-2}{2}\Big| \lesssim |t|^{-\frac{N}{N-6}}.
\end{equation}
For this, we first check that
\begin{equation}
  \label{eq:B3-estim-31}
  |\la \eee^{i\theta}\Lambda W_\lambda, f(\eee^{i\zeta}W_\mu + \eee^{i\theta}W_\lambda) - f(\eee^{i\zeta}W_\mu) - f(\eee^{i\theta}W_\lambda) - f'(\eee^{i\theta}W_\lambda)(\eee^{i\zeta}W_\mu)\ra| \lesssim \lambda^\frac N2.
\end{equation}
  In the region $|x| \geq \sqrt\lambda$ we have $W_\lambda \lesssim W_\mu$, which implies
  \begin{equation}
    |f(\eee^{i\zeta}W_\mu + \eee^{i\theta}W_\lambda) - f(\eee^{i\zeta}W_\mu) - f(\eee^{i\theta}W_\lambda) - f'(\eee^{i\theta}W_\lambda)(\eee^{i\zeta}W_\mu)| \lesssim W_\lambda^\frac{4}{N-2}W_\mu,
  \end{equation}
  hence the required bound follows from $|\Lambda W| \lesssim W$ and
  \begin{equation}
    \|W_\lambda^\frac{N+2}{N-2}\|_{L^1(|x| \geq \sqrt\lambda)} \lesssim \lambda^\frac{N-2}{2}\int_{\lambda^{-\frac 12}}^{+\infty}r^{-N-2}r^{N-1}\ud r \sim \lambda^\frac N2.
  \end{equation}
  In the region $|x| \leq \sqrt\lambda$ we have $W_\mu \lesssim W_\lambda$, which implies
  \begin{equation}
    |f(\eee^{i\zeta}W_\mu + \eee^{i\theta}W_\lambda) - f(\eee^{i\zeta}W_\mu) - f(\eee^{i\theta}W_\lambda) - f'(\eee^{i\theta}W_\lambda)(\eee^{i\zeta}W_\mu)| \lesssim W_\mu^\frac{N+2}{N-2},
  \end{equation}
  hence the required bound follows from
  \begin{equation}
    \|W_\lambda\|_{L^1(|x| \leq \sqrt\lambda)} \lesssim \lambda^\frac{N+2}{2}\int_0^{\lambda^{-\frac 12}}r^{-N+2}r^{N-1}\ud r \sim \lambda^\frac N2.
  \end{equation}

Finally, we need to check that
\begin{equation}
  \label{eq:B3-estim-32}
  \Big|\la \eee^{i\theta}\Lambda W_\lambda, f'(\eee^{i\theta}W_\lambda)(\eee^{i\zeta}W_\mu)\ra - \frac{(N-2)\kappa^\frac{N-4}{2}\|W\|_{L^2}^2}{N-6}\theta\lambda^\frac{N-2}{2}\Big| \lesssim |t|^{-\frac{N}{N-6}}.
\end{equation}
The definition of $f'(z)$ yields
\begin{equation}
  \label{eq:calcul-inter}
  f'(\eee^{i\theta}W_\lambda)(\eee^{i\zeta}W_\mu) = W_\mu W_\lambda^\frac{4}{N-2}\big(\eee^{i\zeta} + \frac{4}{N-2}\eee^{i\theta}\Re(\eee^{i(\zeta - \theta)})\big),
\end{equation}
hence
\begin{equation}
  \label{eq:B3-estim-32-1}
  \la \eee^{i\theta}\Lambda W_\lambda, f'(\eee^{i\theta}W_\lambda)(\eee^{i\zeta}W_\mu)\ra = \frac{N+2}{N-2}\Re(\eee^{i(\zeta- \theta)})\int W_\mu W_\lambda^\frac{4}{N-2}\Lambda W_\lambda\ud x.
\end{equation}
Since $\Big|\int W_\mu W_\lambda^\frac{4}{N-2}\Lambda W_\lambda\ud x\Big| \lesssim \lambda^\frac{N-2}{2} \lesssim |t|^{-\frac{N-2}{N-6}}$, we obtain
\begin{equation}
  \label{eq:B3-estim-32-th}
  \Big|\frac{N+2}{N-2}\Re(\eee^{i(\zeta- \theta)})\int W_\mu W_\lambda^\frac{4}{N-2}\Lambda W_\lambda\ud x + \frac{N+2}{N-2}\theta\int W_\mu W_\lambda^\frac{4}{N-2}\Lambda W_\lambda\ud x\Big| \lesssim |t|^{-\frac{N+1}{N-6}} \ll |t|^{-\frac{N}{N-6}}.
\end{equation}
Next, we prove that
\begin{equation}
  \label{eq:B3-estim-32-mu}
  \Big|\int W_\lambda^\frac{4}{N-2}\Lambda W_\lambda\ud x - \int W_\mu W_\lambda^\frac{4}{N-2}\Lambda W_\lambda\ud x\Big| \lesssim \lambda^\frac N2 + |\mu - 1|\lambda^\frac{N-2}{2} \lesssim |t|^{-\frac{N}{N-6}}.
\end{equation}
Indeed, in the region $|x| \geq \sqrt\lambda$ both terms verify the bound. In the region $|x| \leq \sqrt\lambda$ we have $\big|W_\mu - \mu^{-\frac{N-2}{2}}\big| \lesssim |x|^2 \lesssim \lambda$ and $\big|\mu^{-\frac{N-2}{2}} - 1\big| \lesssim |\mu - 1|$, from which \eqref{eq:B3-estim-32-mu} follows.

From \eqref{eq:explicit-3} and \eqref{eq:explicit-1} we get
\begin{equation}
  \frac{N+2}{N-2}\int W_\lambda^\frac{4}{N-2}\Lambda W_\lambda \ud x = -\frac{(N-2)\kappa^\frac{N-4}{2} \|W\|_{L^2}^2}{N-6}\lambda^\frac{N-2}{2},
\end{equation}
and \eqref{eq:B3-estim-32} follows from \eqref{eq:B3-estim-32-1}, \eqref{eq:B3-estim-32-th} and \eqref{eq:B3-estim-32-mu}.

From \eqref{eq:B3-K}, \eqref{eq:B3-estim-2}, \eqref{eq:B3-estim-3} and the triangle inequality we infer
\begin{equation}
  \label{eq:B3-estim}
  \Big|B_3 - K + \frac{(N-2)\kappa^\frac{N-4}{2}\|W\|_{L^2}^2}{N-6}\theta\lambda^\frac{N-2}{2}\Big| \lesssim |t|^{-\frac{N}{N-6}} + |t|^{-\frac{N}{2(N-6)}}\|g\|_\cE \lesssim |t|^{-\frac{2N-1}{2(N-6)}}.
\end{equation}
In particular, since $|\theta \lambda^\frac{N-2}{2}| \leq |t|^{-\frac{N-1}{N-6}}$, we have
\begin{equation}
  \label{eq:B3-estim-rough}
  \big|B_3\big| \lesssim |t|^{-\frac{N-1}{N-6}} + |t|^{-\frac{N}{N-6}} + |t|^{-\frac{N}{2(N-6)}}\|g\|_\cE + \|g\|_\cE^2 \lesssim |t|^{-\frac{N-1}{N-6}} + C_0^2 |t|^{-\frac{N}{N-6}} \lesssim |t|^{-\frac{N-1}{N-6}}.
\end{equation}
  \textbf{Forth row.}
  Differentiating $\la {-}\eee^{i\theta}W_\lambda, g\ra = 0$ we obtain
  \begin{equation}
    \begin{aligned}
      0 &= \dd t \la -\eee^{i\theta}W_\lambda, g\ra = -\theta'\la i\eee^{i\theta}W_\lambda, g\ra + \frac{\lambda'}{\lambda}\la \eee^{i\theta}\Lambda W_\lambda, g\ra - \la \eee^{i\theta}W_\lambda, \partial_t g\ra \\
      &= \zeta'\la i\eee^{i\theta}W_\lambda, i\eee^{i\zeta} W_\mu\ra - \frac{\mu'}{\mu}\la \eee^{i\theta}W_\lambda, \eee^{i\zeta}\Lambda W_\mu\ra \\
      &+ \theta'\big(\la \eee^{i\theta}W_\lambda, i\eee^{i\theta} W_\lambda\ra-\la i\eee^{i\theta}W_\lambda, g\ra\big) + \frac{\lambda'}{\lambda}\big(\la {-}\eee^{i\theta}W_\lambda, \eee^{i\theta}\Lambda W_\lambda\ra +\la \eee^{i\theta}\Lambda W_\lambda, g\ra\big) \\
      &- \big\la \eee^{i\theta} W_\lambda, i\Delta g + i\big(f(\eee^{i\zeta}W_{\mu} + \eee^{i\theta}W_{\lambda} + g) - f(\eee^{i\zeta}W_\mu) - f(\eee^{i\theta}W_\lambda)\big) \big\ra,
  \end{aligned}
  \end{equation}
  which yields
  \begin{align}
    M_{41} &= \mu^{-2}\la i\eee^{i\theta}W_\lambda, i\eee^{i\zeta} W_\mu\ra= O(\lambda^2) = O(|t|^{-\frac{4}{N-6}}), \\
    M_{42} &= \mu^{-2}\la \eee^{i\theta}W_\lambda, \eee^{i\zeta}\Lambda W_\mu\ra =O(\lambda^2) =  O(|t|^{-\frac{4}{N-6}}), \\
    M_{43} &= \lambda^{-2}\big(\la \eee^{i\theta}W_\lambda, i\eee^{i\theta} W_\lambda\ra-\la i\eee^{i\theta}W_\lambda, g\ra\big) = O(\|g\|_\cE) = O(|t|^{-\frac{N-1}{2(N-6)}}), \\
    M_{44} &= \lambda^{-2}\big(\la {-}\eee^{i\theta}W_\lambda, \eee^{i\theta}\Lambda W_\lambda\ra +\la \eee^{i\theta}\Lambda W_\lambda, g\ra\big) = \|W\|_{L^2}^2 + O(\|g\|_\cE) = \|W\|_{L^2}^2 + O(|t|^{-\frac{N-1}{2(N-6)}}).
  \end{align}

  Let us consider the term
  \begin{equation}
    \begin{aligned}
    B_4 &= \big\la \eee^{i\theta} W_\lambda, i\Delta g + i\big(f(\eee^{i\zeta}W_{\mu} + \eee^{i\theta}W_{\lambda} + g) - f(\eee^{i\zeta}W_\mu) - f(\eee^{i\theta}W_\lambda)\big) \big\ra \\
    &= \big\la \eee^{i\theta}W_\lambda, i\big(f(\eee^{i\zeta}W_{\mu} + \eee^{i\theta}W_{\lambda} + g) - f(\eee^{i\zeta}W_\mu) - f(\eee^{i\theta}W_\lambda) - f'(\eee^{i\theta}W_\lambda)g\big) \big\ra,
    \end{aligned}
  \end{equation}
  where the last equality follows from \eqref{eq:conj-ker-1}.

  First we show that
  \begin{equation}
    \label{eq:B4-estim-1}
    \big|\big\la \eee^{i\theta}W_\lambda, i\big(f(\eee^{i\zeta}W_\mu + \eee^{i\theta}W_\lambda + g) - f(\eee^{i\zeta}W_\mu + \eee^{i\theta}W_\lambda) - f'(\eee^{i\zeta}W_\mu + \eee^{i\theta}W_\lambda)g\big)\big\ra\big| \lesssim \|g\|_\cE^2.
  \end{equation}
  Note that \eqref{eq:bootstrap-zeta} and \eqref{eq:bootstrap-theta} imply that $|\eee^{i\zeta}W_\mu + \eee^{i\theta}W_\lambda| \gtrsim W_\lambda$, hence
  \eqref{eq:pointwise-1} with $z_1 = \eee^{i\zeta}W_\mu + \eee^{i\theta}W_\lambda$ and $z_2 = g$ yields
  \begin{equation}
    \label{eq:B3-estim-11}
    |f(\eee^{i\zeta}W_\mu + \eee^{i\theta}W_\lambda + g) - f(\eee^{i\zeta}W_\mu + \eee^{i\theta}W_\lambda) - f'(\eee^{i\zeta}W_\mu + \eee^{i\theta}W_\lambda)g| \lesssim W_\lambda^{-\frac{N-6}{N-2}}|g|^2.
  \end{equation}
  Using the fact that $|\Lambda W| \lesssim W$ and the H\"older inequality we arrive at \eqref{eq:B4-estim-1}.

  The proof of \eqref{eq:B3-estim-2} yields
  \begin{equation}
    \label{eq:B4-estim-2}
    \big|\big\la \eee^{i\theta}W_\lambda, i\big(f'(\eee^{i\zeta}W_\mu + \eee^{i\theta}W_\lambda) - f'(\eee^{i\theta}W_\lambda)\big)g\big\ra\big| \lesssim \lambda^\frac{N}{4}\|g\|_\cE.
  \end{equation}

  The proof of \eqref{eq:B3-estim-31} yields
  \begin{equation}
    \label{eq:B4-estim-3}
    \big|\big\la \eee^{i\theta}W_\lambda, i\big(f(\eee^{i\zeta}W_\mu + \eee^{i\theta}W_\lambda) - f(\eee^{i\zeta}W_\mu) - f(\eee^{i\theta}W_\lambda) - f'(\eee^{i\theta}W_\lambda)(\eee^{i\zeta}W_\mu)\big)\big\ra\big| \lesssim \lambda^\frac N2.
  \end{equation}

  Finally, we show that
  \begin{equation}
    \label{eq:B4-estim-4}
    \Big|\big\la \eee^{i\theta}W_\lambda, if'(\eee^{i\theta}W_\lambda)(\eee^{i\zeta}W_\mu)\big\ra - \frac{2\kappa^\frac{N-4}{2}\|W\|_{L^2}^2}{N-6}\lambda^\frac{N-2}{2}\Big| \lesssim |t|^{-\frac{N}{N-6}}.
  \end{equation}
  Using again \eqref{eq:calcul-inter} we get
  \begin{equation}
    \label{eq:B4-estim-40}
    \big\la \eee^{i\theta}W_\lambda, if'(\eee^{i\theta}W_\lambda)(\eee^{i\zeta}W_\mu)\big\ra = \Re(i\eee^{i(\zeta-\theta)})\int W_\mu W_\lambda^\frac{N+2}{N-2}\ud x.
  \end{equation}
  We have $\big|\Re(\eee^{-i\theta}) - 1\big| \lesssim |\theta|^2 \leq |t|^{-\frac{2}{N-6}}$ and $\big|i\eee^{i(\zeta - \theta)} - \eee^{-i\theta}\big| = |\eee^{i\zeta} + i| \lesssim |\zeta| \leq |t|^{-\frac{3}{N-6}}$, hence
  \begin{equation}
    \label{eq:B4-estim-41}
    \big|\Re(i\eee^{i(\zeta-\theta)}) - 1\big| \lesssim |t|^{-\frac{2}{N-6}}.
  \end{equation}
  Since $\Big|\int W_\mu W_\lambda^\frac{N+2}{N-2}\ud x\Big| \lesssim \lambda^\frac{N-2}{2} \lesssim |t|^{-\frac{N-2}{N-6}}$, we obtain
\begin{equation}
  \label{eq:B4-estim-42}
  \Big|\Re(i\eee^{i(\zeta- \theta)})\int W_\mu W_\lambda^\frac{N+2}{N-2}\ud x - \int W_\mu W_\lambda^\frac{N+2}{N-2}\ud x\Big| \lesssim |t|^{-\frac{N}{N-6}}.
\end{equation}

  The proof of \eqref{eq:B3-estim-32-mu} yields
  \begin{equation}
  \label{eq:B4-estim-43}
  \Big|\int W_\lambda^\frac{N+2}{N-2}\ud x - \int W_\mu W_\lambda^\frac{N+2}{N-2}\ud x\Big| \lesssim \lambda^\frac N2 + |\mu-1|\lambda^\frac{N-2}{2} \lesssim |t|^{-\frac{N}{N-6}}.
\end{equation}
From \eqref{eq:explicit-2} we get
\begin{equation}
  \int W_\lambda^\frac{N+2}{N-2} \ud x = \frac{2\kappa^\frac{N-4}{2} \|W\|_{L^2}^2}{N-6}\lambda^\frac{N-2}{2},
\end{equation}
hence \eqref{eq:B4-estim-4} follows from \eqref{eq:B4-estim-40}, \eqref{eq:B4-estim-42} and \eqref{eq:B4-estim-43}.
  
  From \eqref{eq:B4-estim-1}, \eqref{eq:B4-estim-2}, \eqref{eq:B4-estim-3}, \eqref{eq:B4-estim-4} and the triangle inequality we obtain
  \begin{equation}
    \label{eq:B4-estim}
    \Big|B_4 - \frac{2\kappa^\frac{N-4}{2}\|W\|_{L^2}^2}{N-6}\lambda(t)^\frac{N-2}{2}\Big| \lesssim |t|^{-\frac{N}{N-6}} + \|g\|_\cE^2,
  \end{equation}
  in particular
  \begin{equation}
    \label{eq:B4-estim-rough}
    |B_4| \lesssim |t|^{-\frac{N-2}{N-6}} + \|g\|_\cE^2 \lesssim |t|^{-\frac{N-2}{N-6}}.
  \end{equation}
  \textbf{Conclusion}
  From the bounds on the coefficients $M_{ij}$ obtained above it follows that the matrix $(M_{ij})$ is strictly diagonally dominant.
  Since $N \geq 7$, \eqref{eq:bootstrap-g} implies that $\|g\|_\cE \lesssim |t|^{-\frac{2}{N-6}}$, hence we can write
  \begin{equation}
    \label{eq:mod-system-approx}
    \begin{gathered}
    \begin{pmatrix}
      M_{11} & M_{12} & M_{13} & M_{14} \\ M_{21} & M_{22} & M_{23} & M_{24} \\ M_{31} & M_{32} & M_{33} & M_{34} \\ M_{41} & M_{42} & M_{43} & M_{44}
    \end{pmatrix}= \\
    \begin{pmatrix}
      \|W\|_{L^2}^2 + O(|t|^{-\frac{2}{N-6}}) & O(|t|^{-\frac{2}{N-6}}) & O(1) & O(1) \\ O(|t|^{-\frac{2}{N-6}}) & \|W\|_{L^2}^2 + O(|t|^{-\frac{2}{N-6}}) & O(1) & O(1) \\ O(|t|^{-\frac{2}{N-6}}) & O(|t|^{-\frac{2}{N-6}}) & \|W\|_{L^2}^2 + O(|t|^{-\frac{2}{N-6}}) & O(|t|^{-\frac{2}{N-6}}) \\ O(|t|^{-\frac{2}{N-6}}) & O(|t|^{-\frac{2}{N-6}}) & O(|t|^{-\frac{2}{N-6}}) & \|W\|_{L^2}^2 + O(|t|^{-\frac{2}{N-6}})
    \end{pmatrix}.
  \end{gathered}
  \end{equation}
  Let $(m_{jk}) = (M_{jk})^{-1}$. It is easy to see that the Cramer's rule implies that $(m_{jk})$ is also of the form given in \eqref{eq:mod-system-approx},
  with $\|W\|_{L^2}^{-2}$ instead of $\|W\|_{L^2}^2$ for the diagonal terms.

  Resuming \eqref{eq:B1-estim}, \eqref{eq:B2-estim}, \eqref{eq:B3-estim-rough} and \eqref{eq:B4-estim-rough}, we have
  \begin{equation}
    \label{eq:B-estim}
    |B_1| + |B_2| + |B_3| + |B_4| \lesssim |t|^{-\frac{N-2}{N-6}}.
  \end{equation}
  This and the form of the matrix $(m_{jk})$ directly imply $|\zeta'| + |\mu'| \lesssim |t|^{-\frac{N-2}{N-6}}$, hence \eqref{eq:mod-zeta} and \eqref{eq:mod-mu}.
  Note that the coefficients in the third and the forth row of the matrix $(m_{jk})$ let us gain an additional factor $|t|^{-\frac{2}{N-6}}$.
  We obtain $\big|\lambda\lambda' - \|W\|_{L^2}^{-2}B_4\big| \lesssim |t|^{-\frac{N-1}{N-6}}$, which implies \eqref{eq:mod-l} thanks to \eqref{eq:B4-estim}.
  Similarly, \eqref{eq:B3-estim} yields \eqref{eq:mod-th}, which finishes the proof.
\end{proof}
\begin{remark}
  A computation similar to the proof of \eqref{eq:B1-estim-1} shows that $|K| \lesssim \|g\|_\cE^2 \leq |t|^{-\frac{N-1}{N-6}}$,
  so we obtain the following simple consequence of Lemma~\ref{lem:mod}:
  \begin{equation}
    \label{eq:param-all}
    |\zeta'(t)| + \Big|\frac{\mu'(t)}{\mu(t)}\Big| + |\theta'(t)| + \Big|\frac{\lambda'(t)}{\lambda(t)}\Big| \lesssim |t|^{-1}
  \end{equation}
  (for the last term, this bound is sharp).
\end{remark}
\subsection{Control of the stable and unstable component}
An important step is to control
the stable and unstable components $a_1^\pm(t) = \la \alpha_{\zeta(t), \mu(t)}^\pm, g(t)\ra$ and $a_2^\pm(t)= \la \alpha_{\theta(t), \lambda(t)}^\pm, g(t)\ra$. Recall that $\nu > 0$ is the positive eigenvalue of the linearized flow, see \eqref{eq:Y1Y2}.
\begin{lemma}
  \label{lem:proper}
  Under assumptions of Lemma~\ref{lem:mod}, for $t \in [T, T_1]$ there holds
  \begin{align}
    \big| \dd t a_1^+(t) - \frac{\nu}{\mu(t)^2}a_1^+(t)\big| &\leq \frac{c}{\mu(t)^2}|t|^{-\frac{N}{2(N-6)}}, \label{eq:proper-1p} \\
    \big| \dd t a_1^-(t) + \frac{\nu}{\mu(t)^2}a_1^-(t)\big| &\leq \frac{c}{\mu(t)^2}|t|^{-\frac{N}{2(N-6)}}, \label{eq:proper-1m} \\
    \big| \dd t a_2^+(t) - \frac{\nu}{\lambda(t)^2}a_2^+(t)\big| &\leq \frac{c}{\lambda(t)^2}|t|^{-\frac{N}{2(N-6)}}, \label{eq:proper-2p} \\
    \big| \dd t a_2^-(t) + \frac{\nu}{\lambda(t)^2}a_2^-(t)\big| &\leq \frac{c}{\lambda(t)^2}|t|^{-\frac{N}{2(N-6)}}, \label{eq:proper-2m}
  \end{align}
  with $c \to 0$ as $|T_0| \to +\infty$.
\end{lemma}
\begin{proof}
  We will give a proof of \eqref{eq:proper-1p} and \eqref{eq:proper-2p}, the other two inequalities being analogous.

  Applying the chain rule to the formula $a_1^+(t) = \la \alpha_{\zeta(t), \mu(t)}^+, g(t)\ra$ and using the definition of $\alpha_{\zeta, \mu}^+$ 
  we obtain
  \begin{equation}
    \dd t a_1^+ = -\frac{\mu'}{\mu}\big\la\frac{\eee^{i\zeta}}{\mu^2}\big(\Lambda_{-1}\cY_\mu^{(2)} + i\Lambda_{-1}\cY_\mu^{(1)}\big), g\big\ra
    + \zeta'\big\la \frac{\eee^{i\zeta}}{\mu^2}\big(i\cY_\mu^{(2)} - \cY_\mu^{(1)}\big), g\big\ra + \la \alpha_{\zeta, \mu}^+, \partial_t g\ra.
  \end{equation}
  Thanks to \eqref{eq:param-all} and \eqref{eq:bootstrap-g},
  the size of the first two terms is $\lesssim |t|^{-1}|t|^{-\frac{N-1}{2(N-6)}} = |t|^{-\frac{3N-13}{2(N-6)}} \ll |t|^{-\frac{N}{2(N-6)}}$.
  We are left with the third term, and we expand $\partial_t g$ according to \eqref{eq:dtg}.

  Let us consider, on by one, the contributions of the four terms in the second line of \eqref{eq:dtg}.
  \begin{enumerate}[1.]
    \item The term $\big\la \alpha_{\zeta, \mu}^+, -\zeta'i\eee^{i\zeta}W_\mu\big\ra$ is equal to $0$ thanks to \eqref{eq:proper-iW}.
    \item The term $\big\la \alpha_{\zeta, \mu}^+, \frac{\mu'}{\mu}\eee^{i\zeta}\Lambda W_\mu\big\ra$ is equal to $0$ thanks to \eqref{eq:proper-LW}.
    \item Consider the term $\big\la \alpha_{\zeta, \mu}^+, -\theta'i\eee^{i\theta}W_\lambda\big\ra$. We have $\|\alpha_{\zeta, \mu}^+\|_{\dot H^1} \lesssim 1$,
      hence
      \begin{equation}
        \big|\big\la \alpha_{\zeta, \mu}^+, -\theta'i\eee^{i\theta}W_\lambda\big\ra\big| \lesssim |\theta'|\|\alpha_{\zeta, \mu}^+\|_{\dot H^1}\|W_\lambda\|_{\dot H^{-1}}
        \lesssim |\theta'|\lambda^2,
      \end{equation}
      and \eqref{eq:param-all} yields $|\theta'|\lambda^2 \lesssim |t|^{-1}|t|^{-\frac{4}{N-6}} = |t|^{-\frac{N-2}{N-6}} \ll |t|^{-\frac{N}{2(N-6)}}$.
    \item The term $\big\la \alpha_{\zeta, \mu}^+, \frac{\lambda'}{\lambda}\eee^{i\theta}\Lambda W_\lambda\big\ra$ is treated as the previous one,
      using $\big|\frac{\lambda'}{\lambda}\big| \lesssim |t|^{-1}$ instead of $|\theta'| \lesssim |t|^{-1}$.
  \end{enumerate}

  Let us finally consider the contribution of the first line of \eqref{eq:dtg}.
  We have
  \begin{equation}
    \begin{gathered}
    i\Delta g + i\big(f(\eee^{i\zeta}W_{\mu} + \eee^{i\theta}W_{\lambda} + g) - f(\eee^{i\zeta}W_\mu) - f(\eee^{i\theta}W_\lambda)\big) = \\
    = Z_{\zeta, \mu}g + i\big(f(\eee^{i\zeta}W_{\mu} + \eee^{i\theta}W_{\lambda} + g) - f(\eee^{i\zeta}W_\mu) - f(\eee^{i\theta}W_\lambda) - f'(\eee^{i\zeta}W_\mu)g\big).
  \end{gathered}
  \end{equation}
  From \eqref{eq:ap-eigen} we obtain $\la\alpha_{\zeta, \mu}^+, Z_{\zeta, \mu}g\ra = \frac{\nu}{\mu^2}a_1^+$, hence we need to show that
  \begin{equation}
    \big|\big\la\alpha_{\zeta, \mu}^+, i\big(f(\eee^{i\zeta}W_{\mu} + \eee^{i\theta}W_{\lambda} + g) - f(\eee^{i\zeta}W_\mu) - f(\eee^{i\theta}W_\lambda) - f'(\eee^{i\zeta}W_\mu)g\big)\big\ra\big| \ll |t|^{-\frac{N}{2(N-6)}}.
  \end{equation}
The proof of \eqref{eq:B1-estim} yields the bound $|t|^{-\frac{N-2}{N-6}} \ll |t|^{-\frac{N}{2(N-6)}}$.

We turn to the proof of \eqref{eq:proper-2p}.
  Applying the chain rule to the formula $a_2^+(t) = \la \alpha_{\zeta(t), \mu(t)}^+, g(t)\ra$ and using the definition of $\alpha_{\theta, \lambda}^+$ 
  we obtain
  \begin{equation}
    \dd t a_2^+ = -\frac{\lambda'}{\lambda}\big\la\frac{\eee^{i\theta}}{\lambda^2}\big(\Lambda_{-1}\cY_\lambda^{(2)} + i\Lambda_{-1}\cY_\lambda^{(1)}\big), g\big\ra
    + \theta'\big\la \frac{\eee^{i\theta}}{\lambda^2}\big(i\cY_\lambda^{(2)} - \cY_\lambda^{(1)}\big), g\big\ra + \la \alpha_{\theta, \lambda}^+, \partial_t g\ra.
  \end{equation}
  The first two terms are treated as in the case of $a_1^+$. In the third term, we expand $\partial_t g$ using \eqref{eq:dtg}.
  Let us consider, on by one, the contributions of the four terms in the second line of \eqref{eq:dtg}.
  \begin{enumerate}[1.]
    \item In order to bound the term $\big\la \alpha_{\theta, \lambda}^+, -\zeta'i\eee^{i\zeta}W_\mu\big\ra$, notice that
      \begin{equation}
        \|\alpha_{\theta, \lambda}^+\|_{L^1} \lesssim \int_{\bR^N}\frac{1}{\lambda^2}\big(|\cY_\lambda^{(1)}| + |\cY_\lambda^{(2)}|\big)\ud x \lesssim \lambda^\frac{N-2}{2} \lesssim |t|^{-\frac{N-2}{N-6}} \ll |t|^{-\frac{N}{2(N-6)}}.
      \end{equation}
      This is sufficient since $\|{-}\zeta'i\eee^{i\zeta}W_\mu\|_{L^\infty} \lesssim 1$.
    \item The term $\big\la \alpha_{\theta, \lambda}^+, \frac{\mu'}{\mu}\eee^{i\zeta}\Lambda W_\mu\big\ra$ is analogous.
    \item The term $\big\la \alpha_{\theta, \lambda}^+, -\theta'i\eee^{i\theta}W_\lambda\big\ra$ is equal to $0$ thanks to \eqref{eq:proper-iW}.
    \item The term $\big\la \alpha_{\theta, \lambda}^+, \frac{\lambda'}{\lambda}\eee^{i\theta}\Lambda W_\lambda\big\ra$ is equal to $0$ thanks to \eqref{eq:proper-LW}.
  \end{enumerate}

  Let us finally consider the contribution of the first line of \eqref{eq:dtg}.
  We have
  \begin{equation}
    \begin{gathered}
    i\Delta g + i\big(f(\eee^{i\zeta}W_{\mu} + \eee^{i\theta}W_{\lambda} + g) - f(\eee^{i\zeta}W_\mu) - f(\eee^{i\theta}W_\lambda)\big) = \\
    = Z_{\theta, \lambda}g + i\big(f(\eee^{i\zeta}W_{\mu} + \eee^{i\theta}W_{\lambda} + g) - f(\eee^{i\zeta}W_\mu) - f(\eee^{i\theta}W_\lambda) - f'(\eee^{i\theta}W_\lambda)g\big).
  \end{gathered}
  \end{equation}
  From \eqref{eq:ap-eigen} we obtain $\la\alpha_{\theta, \lambda}^+, Z_{\theta, \lambda}g\ra = \frac{\nu}{\lambda^2}a_2^+$, hence we need to show that
  \begin{equation}
    \label{eq:a2p-final}
    \lambda^2\big|\big\la\alpha_{\theta, \lambda}^+, i\big(f(\eee^{i\zeta}W_{\mu} + \eee^{i\theta}W_{\lambda} + g) - f(\eee^{i\zeta}W_\mu) - f(\eee^{i\theta}W_\lambda) - f'(\eee^{i\theta}W_\lambda)g\big)\big\ra\big| \ll |t|^{-\frac{N}{2(N-6)}}.
  \end{equation}
  The proof of \eqref{eq:B3-estim-2} yields
  \begin{equation}
    \label{eq:a2p-final1}
    \begin{gathered}
    \lambda^2\big|\big\la\alpha_{\theta, \lambda}^+, i(f'(\eee^{i\zeta}W_\mu + \eee^{i\theta}W_\lambda)-f'(\eee^{i\theta}W_\lambda))g\big\ra\big| \lesssim \lambda^\frac N4\|g\|_\cE \\
    \lesssim |t|^{-\frac{N}{2(N-6)} - \frac{N-1}{2(N-6)}} \ll |t|^{-\frac{N}{2(N-6)}}.
  \end{gathered}
  \end{equation}
  The proof of \eqref{eq:B4-estim-1} yields
  \begin{equation}
    \label{eq:a2p-final2}
    \lambda^2\big|\big\la\alpha_{\theta, \lambda}^+, i\big(f(\eee^{i\zeta}W_\mu + \eee^{i\theta}W_\lambda + g) - f(\eee^{i\zeta}W_\mu + \eee^{i\theta}W_\lambda)
    - f'(\eee^{i\zeta}W_\mu + \eee^{i\theta}W_\lambda)g\big)\big\ra\big| \lesssim \|g\|_\cE^2 \ll |t|^{-\frac{N}{2(N-6)}}.
  \end{equation}
  Using \eqref{eq:pointwise-2} we get
  \begin{equation}
    \|f(\eee^{i\zeta}W_\mu + \eee^{i\theta}W_\lambda) - f(\eee^{i\zeta}W_\mu) - f(\eee^{i\theta}W_\lambda)\|_{L^\infty} \lesssim \|W_\lambda^\frac{4}{N-2}W_\mu\|_{L^\infty} \lesssim \frac{1}{\lambda^2}.
  \end{equation}
  By a change of variable, $\|\lambda^2 \alpha_{\theta, \lambda}^+\|_{L^1} \lesssim \lambda^\frac{N+2}{2}$, hence
  \begin{equation}
    \label{eq:a2p-final3}
    \lambda^2\big|\big\la\alpha_{\theta, \lambda}^+, i\big(f(\eee^{i\zeta}W_\mu + \eee^{i\theta}W_\lambda) - f(\eee^{i\zeta}W_\mu) - f(\eee^{i\theta}W_\lambda)\big)\big\ra\big| \lesssim \lambda^{\frac{N-2}{2}} \lesssim |t|^{-\frac{N-2}{N-6}} \ll |t|^{-\frac{N}{2(N-6)}}.
  \end{equation}
  Taking the sum of \eqref{eq:a2p-final1}, \eqref{eq:a2p-final2} and \eqref{eq:a2p-final3} and using the triangle inequality, we obtain \eqref{eq:a2p-final}.
\end{proof}

\section{Bootstrap}
\label{sec:boot}
We turn to the heart of the proof, which consists in establishing bootstrap estimates.
We consider a solution $u(t)$, decomposed according to \eqref{eq:decompose}, \eqref{eq:param-rough} and \eqref{eq:orth}.
The initial data at time $T \leq T_0$ is chosen as follows.
\begin{lemma}
  \label{lem:initial}
  There exists $T_0 < 0$ such that for all $T \leq T_0$ and for all $\lambda^0$, $a_1^0$, $a_2^0$
  satisfying
  \begin{equation}
    \label{eq:initial-assum}
    \big|\lambda^0 - \frac{1}{\kappa}(\kappa|T|)^{-\frac{2}{N-6}}\big| \leq \frac 12 |T|^{-\frac{5}{2(N-6)}},\qquad |a_1^0| \leq \frac 12 |T|^{-\frac{N}{2(N-6)}},\qquad |a_2^0| \leq \frac 12 |T|^{-\frac{N}{2(N-6)}},
  \end{equation}
  there exists $g^0 \in X^1$ satisfying
  \begin{gather}
    \label{eq:initial-orth}
    \la \Lambda W, g^0\ra = \la iW, g^0\ra = \la i\Lambda W_{\lambda^0}, g^0\ra = \la {-}W_{\lambda^0}, g^0\ra = 0, \\
    \label{eq:initial-unstable}
    \la \alpha_{-\frac{\pi}{2},1}^-, g^0\ra = 0,\quad \la \alpha_{-\frac{\pi}{2},1}^+, g^0\ra = a_1^0,\quad 
\la \alpha_{0,\lambda^0}^-, g^0\ra = 0,\quad \la \alpha_{0, \lambda^0}^+, g^0\ra = a_2^0, \\
    \label{eq:initial-size}
    \|g^0\|_{\cE} \lesssim |T|^{-\frac{N}{2(N-6)}}.
  \end{gather}
  This $g^0$ is continuous for the $X^1$ topology with respect to $\lambda^0$, $a_1^0$ and $a_2^0$.
\end{lemma}
\begin{remark}
  For the continuity, we just claim that the function $g^0$ constructed in the proof
  is continuous with respect to $\lambda^0$, $a_1^0$ and $a_2^0$.
  Clearly, $g^0$ is not uniquely determined by \eqref{eq:initial-orth}, \eqref{eq:initial-unstable} and \eqref{eq:initial-size}.
\end{remark}
\begin{remark}
  Condition \eqref{eq:initial-orth} is exactly \eqref{eq:orth} with $\big(\zeta, \mu, \theta, \lambda\big) = \big(-\frac{\pi}{2}, 1, 0, \lambda^0\big)$.
  Hence, if we consider the solution $u(t)$ of \eqref{eq:nls} with initial data $u(T) = -iW + W_{\lambda^0} + g^0$
  and decompose it according to \eqref{eq:decompose}, then $g(T) = g^0$ and the initial values of the modulation parameters are
  $\big(\zeta(T), \mu(T), \theta(T), \lambda(T)\big) = \big({-}\frac{\pi}{2}, 1, 0, \lambda^0\big)$.
\end{remark}
\begin{proof}
  We consider functions of the form
  \begin{equation}
    g^0 = a_1^+i\alpha_{-\frac{\pi}{2}, 1}^- - a_1^- i\alpha_{-\frac{\pi}{2}, 1}^+ + b_1 W + c_1 (-i\Lambda W) + a_2^+(\lambda^0)^2i\alpha_{0, \lambda^0}^- - a_2^-(\lambda^0)^2i\alpha_{0, \lambda^0}^+ + b_2 iW_{\lambda^0} + c_2 \Lambda W_{\lambda^0},
  \end{equation}
  with $a_1^+$, $a_1^-$, $b_1$, $c_1$, $a_2^+$, $a_2^-$, $b_2$, $c_2$ being real numbers.
  Let $\Phi: \bR^8 \to \bR^8$ be the linear map defined as follows:
  \begin{equation}
  \begin{gathered}
    \Phi(a_1^+, a_1^-, b_1, c_1, a_2^+, a_2^-, b_2, c_2) := \\
    \big(\la \alpha_{-\frac{\pi}{2}, 1}^+, g^0\ra, \la \alpha_{-\frac{\pi}{2}, 1}^-, g^0\ra, \la \Lambda W, g^0\ra, \la iW, g^0\ra, \\
    \la \alpha_{0, \lambda^0}^+, g^0\ra, \la \alpha_{0, \lambda^0}^-, g^0\ra, \big\la (\lambda^0)^{-2}i\Lambda W_{\lambda^0}, g^0\big\ra, \big\la {-}(\lambda^0)^{-2}W_{\lambda^0}, g^0\big\ra\big).
    \end{gathered}
  \end{equation}
  Using \eqref{eq:proper-iW}, \eqref{eq:proper-LW}, \eqref{eq:Y1Y2-prod} and the fact that $\lambda^0$ is small
  we obtain that the matrix of $\Phi$ is strictly diagonally dominant, which implies the result.
\end{proof}
In the remaining part of this section, we will analyze solutions $u(t)$ of \eqref{eq:nls}
with the initial data $u(T) = -iW + W_{\lambda^0} + g^0$,
where $g^0$ is given by the previous lemma. \begin{proposition}
  \label{prop:bootstrap}
  There exists $T_0 <0$ with the following property.
  Let $T < T_1 < T_0$ and let $\lambda^0, a_1^0, a_2^0$ satisfy \eqref{eq:initial-assum}.
  Let $g^0 \in X^1$ be given by Lemma~\ref{lem:initial} and consider the solution $u(t)$ of \eqref{eq:nls}
  with the initial data $u(T) = -iW + W_{\lambda^0} + g^0$.
  Suppose that $u(t)$ exists on the time interval $[T, T_1]$, that for $t \in [T, T_1]$
  conditions \eqref{eq:bootstrap-zeta}, \eqref{eq:bootstrap-mu}, \eqref{eq:bootstrap-theta}, \eqref{eq:bootstrap-lambda}
  and \eqref{eq:bootstrap-g} hold, and moreover that
\begin{equation}
  \label{eq:bootstrap-unstable}
  |a_1^+(t)| \leq |t|^{-\frac{N}{2(N-6)}},\qquad |a_2^+(t)| \leq |t|^{-\frac{N}{2(N-6)}}.
  \end{equation}
  Then for $t \in [T, T_1]$ there holds
  \begin{align}
    \big|\zeta(t) + \frac{\pi}{2}\big| &\leq \frac 12 |t|^{-\frac{3}{N-6}}, \label{eq:bootstrap-better-zeta} \\
    |\mu(t) - 1| &\leq \frac 12 |t|^{-\frac{3}{N-6}},  \label{eq:bootstrap-better-mu} \\
    |\theta(t)| &\leq \frac 12 |t|^{-\frac{1}{N-6}}, \label{eq:bootstrap-better-theta} \\
    \|g(t)\|_\cE &\leq \frac 12 |t|^{-\frac{N-1}{2(N-6)}}. \label{eq:bootstrap-better-g}
  \end{align}
\end{proposition}
Before we give a proof, we need a little preparation.
\subsection{A virial-type correction}
The delicate part of the proof of Proposition~\ref{prop:bootstrap} will be to control $\theta(t)$. For this, we will need to use a virial functional, which we now define.

\begin{lemma}
  \label{lem:fun-a}
  For any $c > 0$ and $R > 0$ there exists a radial function $q(x) = q_{c,R}(x) \in C^{3,1}(\bR^N)$ with the following properties:
  \begin{enumerate}[label=(P\arabic*)]
    \item $q(x) = \frac 12 |x|^2$ for $|x| \leq R$, \label{enum:approx}
    \item there exists $\wt R > 0$ (depending on $c$ and $R$) such that $q(x) \equiv \tx{const}$ for $|x| \geq \wt R$, \label{enum:support}
    \item $|\grad q(x)| \lesssim |x|$ and $|\Delta q(x)| \lesssim 1$ for all $x \in \bR^N$, with constants independent of $c$ and $R$, \label{enum:gradlap}
    \item $\sum_{1\leq j, k\leq N} \big(\partial_{x_j x_k} q(x)\big) \conj{v_j} v_k \geq -c\sum_{j=1}^N |v_j|^2$, for all $x \in \bR^N, v_j \in \bC$, \label{enum:convex}
    \item $\Delta^2 q(x) \leq c\cdot|x|^{-2}$, for all $x \in \bR^N$. \label{enum:bilapl}
  \end{enumerate}
\end{lemma}
\begin{remark}
  We require $C^{3, 1}$ regularity in order not to worry about boundary terms in Pohozaev identities, see the proof of \eqref{eq:A-pohozaev}.
\end{remark}
\begin{proof}
  It suffices to prove the result for $R = 1$ since the function $q_R(x) := R^2 q(\frac xR)$ satisfies
  the listed properties if and only if $q(x)$ does.

  Let $r$ denote the radial coordinate. Define $q_0(x)$ by the formula
  \begin{equation}
    \label{eq:q0}
    q_0(r) := \left\{
      \begin{aligned}
        &\frac 12 r^2\qquad & r\leq 1 \\
        &\frac{N(N-2)r}{(N-1)(N-3)} - \frac{N}{2(N-4)}+ \frac{N}{2(N-3)(N-4)r^{N-4}} - \frac{1}{2(N-1)r^{N-2}}\qquad & r \geq 1.
      \end{aligned}\right.
  \end{equation}
  A direct computation shows that for $r > 1$ we have $q_0'(r) = \frac{N(N-2)}{(N-1)(N-3)} - \frac{N}{2(N-3)r^{N-3}} + \frac{N-2}{2(N-1)r^{N-1}}$,
  $q_0''(r) = \frac{N}{2r^{N-2}} - \frac{N-2}{2r^N} > 0$ (so $q_0(x)$ is convex), $q_0'''(r) = \frac{N(N-2)}{2}\big({-}\frac{1}{r^{N-1}}+\frac{1}{r^{N+1}}\big)$ and $\Delta^2 q_0(r) = -N(N-2) r^{-3} < 0$.
  In particular,
  \begin{equation}
  \lim_{r \to 1^+} \big(q_0(r), q_0'(r), q_0''(r), q_0'''(r)\big) = \big(\frac 12, 1, 1, 0\big).
\end{equation}
Hence $q_0 \in C^{3,1}$ and it satisfies all the listed properties except for \ref{enum:support}. We correct it as follows.

  Let $e_j(r) := \frac{1}{j!}r^j\cdot \chi(r)$ for $j \in \{1, 2, 3\}$ and let $R_0 \gg 1$. We define
  \begin{equation}
    \label{eq:q}
    q(r) := \bigg\{
      \begin{aligned}
        &q_0(r)\qquad & r\leq R_0 \\
        &{\textstyle q_0(R_0) + \sum_{j=1}^{3}q_0^{(j)}(R_0)\cdot R_0^j\cdot e_j(-1 + R_0^{-1}r)}\qquad & r \geq R_0.
      \end{aligned}
  \end{equation}
  Note that $q_0'(R_0) \sim 1$, $q_0''(R_0) \sim R_0^{-N+2}$ and $q_0'''(R_0) \sim R_0^{-N+1}$.
  It is clear that $q(x) \in C^{3,1}(\bR^N)$.
  Property \ref{enum:approx} holds since $R_0 > 1$. By the definition of the functions $e_j$ we have $q(r) = q_0(R_0) = \tx{const}$ for $r \geq 3R_0$,
  hence \ref{enum:support} holds with $\wt R = 3R_0$. From the definition of $q(r)$ we get $|q'(r)| \lesssim |q_0'(R_0)| \lesssim r$ and $|q''(r)| \lesssim |q_0''(R_0)| \lesssim R_0^{-N+2} \lesssim 1$ for $r \geq R_0$,
  with a constant independent of $R_0$, which implies \ref{enum:gradlap}. Similarly, $|\partial_{x_i x_j}q(x)| \lesssim R_0^{-1}$ for $|x| \geq R_0$, which implies \ref{enum:convex}
  if $R_0$ is large enough. Finally $|\Delta^2 q(x)| \lesssim R_0^{-3}$ for $|x| \geq R_0$ and $\Delta^2 q(x) = 0$ for $|x| \geq 3R_0$. This proves \ref{enum:bilapl} if $R_0$ is large enough.
\end{proof}
In the sequel $q(x)$ always denotes a function of class $C^{3, 1}(\bR^N)$ verifying \ref{enum:approx}--\ref{enum:bilapl}
with sufficiently small $c$ and sufficiently large $R$.

For $\lambda > 0$ we define the operators $A(\lambda)$ and $A_0(\lambda)$ as follows.
\begin{align}
  \label{eq:op-A}
  [A(\lambda)h](x) &:= \frac{N-2}{2N\lambda^2}\Delta q\big(\frac{x}{\lambda}\big)h(x) + \frac{1}{\lambda}\grad q\big(\frac{x}{\lambda}\big)\cdot \grad h(x), \\
  [A_0(\lambda)h](x) &:= \frac{1}{2\lambda^2}\Delta q\big(\frac{x}{\lambda}\big)h(x) + \frac{1}{\lambda}\grad q\big(\frac{x}{\lambda}\big)\cdot \grad h(x). \label{eq:op-A0}
\end{align}
Combining these definitions with the fact that $q(x)$ is an approximation of $\frac 12 |x|^2$
we see that $A(\lambda)$ and $A_0(\lambda)$ are approximations (in a sense not yet precised)
of $\lambda^{-2}\Lambda$ and $\lambda^{-2}\Lambda_0$ respectively.
We will write $A$ and $A_0$ instead of $A(1)$ and $A_0(1)$ respectively. Note the following scale-change formulas, which follow directly from the definitions:
\begin{equation}
\label{eq:A-rescale}
\forall h\in \cE:\qquad A(\lambda)(h_\lambda) = \lambda^{-2}(Ah)_\lambda,\quad A_0(\lambda)(h_\lambda) = \lambda^{-2}(A_0 h)_\lambda.
\end{equation}
\begin{lemma}
  \label{lem:op-A}
  The operators $A(\lambda)$ and $A_0(\lambda)$ have the following properties:
  \begin{itemize}
    \item for $\lambda > 0$ the families $\{A(\lambda)\}$, $\{A_0(\lambda)\}$, $\{\lambda\partial_\lambda A(\lambda)\}$, $\{\lambda\partial_\lambda A_0(\lambda)\}$
      are bounded in $\scrL(\cE; \dot H^{-1})$ and the families $\{\lambda A(\lambda)\}$, $\{\lambda A_0(\lambda)\}$ are bounded in $\scrL(\cE; L^2)$,
      with the bound depending on the choice of the function $q(x)$,
    \item for all complex-valued $h_1, h_2 \in X^1(\bR^N)$ and $\lambda > 0$ there holds
        \begin{gather}
        \label{eq:A-by-parts}
        \la A(\lambda)h_1, f(h_1 + h_2) - f(h_1) - f'(h_1)h_2\ra = -\la A(\lambda)h_2, f(h_1+h_2) - f(h_1)\ra, \\
        \la h_1, A_0(\lambda)h_2\ra = -\la A_0(\lambda)h_1, h_2\ra, \qquad \text{hence $iA_0(\lambda)$ is a symmetric operator,} \label{eq:A0-by-parts}
      \end{gather}
    \item for any $c_0 > 0$, if we choose $c$ in Lemma~\ref{lem:fun-a} small enough, then for all $h \in X^1$ there holds
      \begin{equation}
        \label{eq:A-pohozaev}
        \la A_0(\lambda)h, \Delta h\ra \leq \frac{c_0}{\lambda^2} \|h\|_{\cE}^2 - \frac{1}{\lambda^2}\int_{|x| \leq R\lambda}|\grad h(x)|^2 \ud x.
      \end{equation}
  \end{itemize}
\end{lemma}
In dimension $N = 6$ and for real-valued functions, this was proved in \cite[Lemma 3.12]{moi16p}.
Most arguments apply without change, but we provide here a full computation for the reader's convenience.
\begin{proof}
  Since $\grad q(x)$ and $\grad^2 q(x)$ are continuous and of compact support, it is clear that $A$ and $A_0$ are bounded operators $\cE \to \dot H^{-1}$.
  From the invariance \eqref{eq:A-rescale} we see that $A(\lambda)$ and $A_0(\lambda)$ have the same norms as $A$ and $A_0$ respectively.
  For $\lambda A(\lambda)$, $\lambda A_0(\lambda)$, $\lambda \partial_\lambda A(\lambda)$ and $\lambda \partial_\lambda A_0(\lambda)$ the proof is similar. We compute
  $$
  \partial_\lambda A(\lambda) = -\frac{N-2}{N\lambda^3}\Delta q\big(\frac{x}{\lambda}\big) - \frac{N-2}{2N\lambda^4}x\cdot\grad \Delta q\big(\frac{x}{\lambda}\big)
  - \frac{1}{\lambda^3}x\cdot\grad^2 q\big(\frac{x}{\lambda}\big)\cdot\grad.
  $$
  Since $\grad q(x)$, $\grad^2 q(x)$ and $\grad^3 q(x)$ are continuous and of compact support, we get boundedness of $\partial_\lambda A(1)$, and boundedness $\{\lambda\partial_\lambda A(\lambda)\}$ follows by the scaling invariance. Analogously for $\{\lambda\partial_\lambda A_0(\lambda)\}$.

  In \eqref{eq:A-by-parts}, we may assume without loss of generality that $\lambda = 1$.
  Notice that both sides are continuous with respect to the topology $\|h_1\|_{X^1} + \|h_2\|_{X^1}$.
  Indeed, $A$ is continuous from $X^1$ to $\cE$ and $(h_1, h_2) \mapsto \big(f(h_1 + h_2) - f(h_1) - f'(h_1)h_2, f(h_1 + h_2) - f(h_1)\big)$ is continuous from $\cE$ to $\dot H^{-1}$ by Sobolev and dual Sobolev.
  We may therefore assume that $h_1, h_2 \in C_0^\infty$.
  Observe that for any $h \in C_0^\infty$ there holds $f(h)\conj h = \frac{2N}{N-2}F(h)$ and $\Re\big(f(h)\grad \conj h\big) = \grad \big(F(h)\big)$, hence
  \begin{equation}
    \label{eq:A-by-parts-2}
    \la Ah, f(h)\ra =\Re \int_{\bR^N}\Big(\frac{N-2}{2N} \Delta q\conj h + \grad q\cdot \grad \conj h\Big)f(h)\ud x = \int_{\bR^N} \Delta q\cdot F(h) + \grad q\cdot \grad \big(F(h)\big)\ud x = 0.
  \end{equation}
  Using this for $h = h_1 + h_2$ and for $h = h_1$, \eqref{eq:A-by-parts} is seen to be equivalent to
  \begin{equation}
    \label{eq:A-by-parts-3}
    \la A h_2, f(h_1)\ra + \la A h_1, f'(h_1)h_2\ra = 0.
  \end{equation}
  Expanding the left side using the definition of $A$ we obtain
  \begin{equation}
    \label{eq:A-by-parts-4}
    \begin{aligned}
      \la A h_2, f(h_1)\ra + \la A h_1, f'(h_1)h_2\ra &= \Re\int_{\bR^N} \frac{N-2}{2N} \Delta q\cdot \conj{h_2}\cdot f(h_1) + \grad q\cdot \grad \conj{h_2}\cdot f(h_1)\ud x \\
      &+\Re\int_{\bR^N} \frac{N-2}{2N} \Delta q\cdot \conj{h_1}\cdot f'(h_1) h_2 + \grad q\cdot \grad \conj{h_1}\cdot f'(h_1) h_2\ud x
    \end{aligned}
  \end{equation}
  We have
  \begin{equation}
    \Re \int_{\bR^N} \grad q\cdot \grad \conj{h_2}\cdot f(h_1)\ud x = -\Re\int_{\bR^N}\conj{h_2}\cdot \Delta q\cdot f(h_1)\ud x
    - \Re\int_{\bR^N} \conj{h_2}\cdot \grad q\cdot \grad{f(h_1)}\ud x.
  \end{equation}
  Using \eqref{eq:auto-scalar} and the fact that $f'(h_1)h_1 = \frac{N+2}{N-2}f(h_1)$ we get
  \begin{equation}
    \Re\int_{\bR^N} \frac{N-2}{2N} \Delta q\cdot \conj{h_1}\cdot f'(h_1) h_2 \ud x
    = \Re \int_{\bR^N} \frac{N-2}{2N}\conj{h_2}\cdot \Delta q\cdot f'(h_1)h_1 = \Re \int_{\bR^N}\conj{h_2}\cdot \frac{N+2}{2N}\Delta q\cdot f(h_1)\ud x.
  \end{equation}
  Using \eqref{eq:auto-scalar} and the fact that $f'(h_1)\grad h_1 = \grad(f(h_1))$ we get
  \begin{equation}
    \Re\int_{\bR^N} \grad q\cdot \grad \conj{h_1}\cdot f'(h_1) h_2\ud x = \Re\int_{\bR^N} \conj{h_2}\cdot \grad q\cdot f'(h_1)\grad h_1\ud x = \Re\int_{\bR^N} \conj{h_2}\cdot \grad q\cdot \grad(f(h_1))\ud x.
  \end{equation}
  Plugging the last three formulas into \eqref{eq:A-by-parts-4} we obtain
  \begin{equation}
    \begin{gathered}
      \la A h_2, f(h_1)\ra + \la A h_1, f'(h_1)h_2\ra = \\
      = \Big\la h_2, \frac{N-2}{2N}\Delta q\cdot f(h_1) - \Delta q\cdot f(h_1) - \grad q\cdot \grad(f(h_1)) + \frac{N+2}{2N}\Delta q\cdot f(h_1)+\grad q\cdot \grad(f(h_1))\Big\ra = \\ = \la h_2, 0\ra = 0,
    \end{gathered}
  \end{equation}
  which proves \eqref{eq:A-by-parts-3}.

  Identity \eqref{eq:A0-by-parts} follows by an integration by parts.

  In \eqref{eq:A-pohozaev}, we can again assume that $\lambda = 1$ and $h \in C_0^\infty$ (we use the fact
  that $q \in C^{3,1}$, hence $\Delta^2 q$ is bounded and of compact support).
  Inequality \eqref{eq:A-pohozaev} follows easily from \ref{enum:approx}, \ref{enum:convex} and \ref{enum:bilapl},
  once we check the following identity:
  \begin{equation}
    \label{eq:aux-pohozaev}
      \Re\int_{\bR^N} \Delta h\cdot\big(\frac{1}{2}\Delta q\cdot\conj h + \grad q\cdot \grad \conj h\big)\ud x
      = -\frac{1}{4}\int_{\bR^N}(\Delta^2 q)|h|^2 \ud x - \int_{\bR^N}\sum_{i, j = 1}^N\partial_{ij}q\partial_i \conj h\partial_j h\ud x.
  \end{equation}
  We can assume that $q \in C_0^\infty$, and \eqref{eq:aux-pohozaev} follows from integration by parts:
  \begin{equation}
    \begin{aligned}
      &\Re\int_{\bR^N} \frac 12 \Delta h\cdot \Delta q\cdot \conj h + \Delta h\cdot \grad q\cdot \grad \conj h\ud x =
      \Re\int_{\bR^N}\sum_{j, k = 1}^N \big( \frac 12 \partial_{jj}h\cdot\partial_{kk}q\cdot \conj h + \partial_{jj}h\cdot \partial_k q\cdot \partial_k \conj h \big)\ud x \\
      &= \Re\int_{\bR^N}-\frac 12 \sum_{j, k}\partial_j h(\partial_{kk}q\partial_j \conj h + \partial_{jkk}q \cdot \conj h)+\sum_j \frac 12 \partial_j(|\partial_j h|^2)\partial_j q \\
      &+ \sum_{j\neq k}\big({-}\frac 12 \partial_k|\partial_j h|^2 \partial_k q - \partial_{jk}q\partial_j \conj h\partial_k h\big)\ud x \\
      &= \Re \int_{\bR^N}-\frac 12 \sum_{j, k}\big(\partial_{kk}q|\partial_j h|^2 + \frac 12 \partial_{jjkk}q\cdot |h|^2\big) - \frac 12 \sum_j \partial_{jj}q|\partial_j h|^2 \\
      &+\frac 12 \sum_{j\neq k}\partial_{kk}q|\partial_j h|^2 - \sum_{j\neq k}\partial_{jk}q\partial_j \conj h\partial_k h\ud x \\
      &= \int_{\bR^N} -\frac 14 \sum_{j, k}\partial_{jjkk}q\cdot |h|^2 -\sum_{j, k}\partial_{jk}q\partial_j \conj h\partial_k h\ud x.
    \end{aligned}
  \end{equation}
\end{proof}
\subsection{Closing the bootstrap}
\begin{proof}[Proof of Proposition~\ref{prop:bootstrap}]
  We split the proof into three steps. First we prove \eqref{eq:bootstrap-better-zeta} and \eqref{eq:bootstrap-better-mu}.
  Then we use the virial functional and variational estimates to prove \eqref{eq:bootstrap-better-theta},
  with $\frac 12$ replaced by any strictly positive constant.
  To do this, we have to deal somehow with the term $\|W\|_{L^2}^{-2}K$ in
  the modulation equation \eqref{eq:mod-th}. It involves terms quadratic in $g$, which is the critical size and will not allow to recover the small constant.
  However, it turns out that we can use a virial functional to absorb the essential part of $K$.
  Proving \eqref{eq:bootstrap-better-theta} is the most difficult step. Finally, \eqref{eq:bootstrap-better-g} will follow from variational estimates.

  \textbf{Step 1.}
  Integrating \eqref{eq:mod-zeta} on $[T, t]$ and using the fact that $\zeta(T) = -\frac{\pi}{2}$ we get
  \begin{equation}
    \big|\zeta(t) + \frac{\pi}{2}\big| = \big|\zeta(t) - \zeta(T)\big| = \big|\int_T^t \zeta'(\tau)\ud \tau\big| \leq c\int_T^t|\tau|^{-\frac{N-3}{N-6}} \leq c\cdot \frac{N-6}{3}|t|^{-\frac{3}{N-6}} \leq \frac 12 |t|^{-\frac{3}{N-6}},
  \end{equation}
  provided that $c \leq \frac{3}{2(N-6)}$.
  The proof of \eqref{eq:bootstrap-better-mu} is similar.

  \textbf{Step 2.}
  First, let us show that for $t \in [T, T_1]$ there holds
  \begin{equation}
    \label{eq:bootstrap-stable}
    |a_1^-(t)| < |t|^{-\frac{N}{2(N-6)}}, \qquad |a_2^-(t)| < |t|^{-\frac{N}{2(N-6)}}.
  \end{equation}
  This is verified initially, see \eqref{eq:initial-unstable}. Suppose that $T_2 \in (T, T_1)$ is the last time for which \eqref{eq:bootstrap-stable} holds for $t \in [T, T_2)$.
    Let for example $a_1^-(T_2) = |T_2|^{-\frac{N}{2(N-6)}}$. But since $\|g(T_2)\|_\cE^2 \lesssim |T_2|^{-\frac{N-1}{N-6}} \ll |T_2|^{-\frac{N}{2(N-6)}}$,
    \eqref{eq:proper-2m} implies that $\dd t a_1^-(T_2) < 0$, which contradicts the assumption that $a_1^-(t) < |T_2|^{-\frac{N}{2(N-6)}}$ for $t < T_2$.
    The proof of the other inequality is similar.

  Let $c_0 > 0$. We will prove that if $T_0$ is chosen large enough (depending on $c_0$), then
  \begin{equation}
    \label{eq:bootstrap-bbetter-theta}
    |\theta(t)| \leq c_0|t|^{-\frac{1}{N-6}},\qquad \text{for }t \in [T, T_1].
  \end{equation}
  By the conservation of energy, \eqref{eq:coer-bound} and \eqref{eq:initial-size} we have
  \begin{equation}
    \label{eq:bootstrap-energy}
    \big|E(u) - 2E(W)\big| = \big|E(u(T)) - 2E(W)\big| \lesssim |T|^{-\frac{N}{N-6}} \leq |t|^{-\frac{N}{N-6}},
  \end{equation}
  hence \eqref{eq:coer-conclusion} yields
  \begin{equation}
    \label{eq:bootstrap-bbetter-theta-leq}
    \theta \lambda^\frac{N-2}{2} \lesssim |t|^{-\frac{N}{N-6}} \quad\Rightarrow\quad \theta \lesssim |t|^{-\frac{N}{N-6} + \frac{N-2}{N-6}} = |t|^{-\frac{2}{N-6}} \ll |t|^{-\frac{1}{N-6}}.
  \end{equation}
  It remains to prove that 
  \begin{equation}
    \label{eq:bootstrap-bbetter-theta-geq}
  \theta \geq -c_0|t|^{-\frac{1}{N-6}}.
\end{equation}
  To this end, we consider the following real scalar function:
\begin{equation}
  \label{eq:psi}
  \psi(t) := \theta(t) - \frac{1}{2\|W\|_{L^2}^2}\la g(t), i A_0(\lambda(t))g(t)\ra.
\end{equation}
We will show that for $t \in [T, T_1]$ there holds
\begin{equation}
  \label{eq:deriv-psi}
  \psi'(t) \geq -c_1|t|^{-\frac{N-5}{N-6}},
\end{equation}
with $c_1 > 0$ as small as we like, by eventually enlarging $|T_0|$.

From \eqref{eq:bootstrap-bbetter-theta-leq} we get $\theta \lambda^{-\frac{N-6}{2}} \ll |t|^{-\frac{N-5}{N-6}}$,
hence, taking in Lemma~\ref{lem:mod} say $c = \frac 14 c_1$ and choosing $|T_0|$ large enough, \eqref{eq:mod-th} yields
\begin{equation}
  \label{eq:deriv-psi-1}
  \begin{aligned}
    \psi' &\geq -\frac{(N-2)\kappa^\frac{N-4}{2}}{N-6}\theta\lambda^\frac{N-6}{2} + \frac{K}{\lambda^2\|W\|_{L^2}^2} - \frac{c_1}{4}|t|^{-\frac{N-5}{N-6}}
  + \frac{1}{2\|W\|_{L^2}^2}\dd t\la g, i A_0(\lambda)g\ra \\
  &\geq \frac{1}{\|W\|_{L^2}^2}\Big(\frac{1}{\lambda^2}K - \frac 12 \dd t\la g, i A_0(\lambda)g\ra\Big) - \frac{c_1}{2}|t|^{-\frac{N-5}{N-6}},
\end{aligned}
\end{equation}
so we need to compute $\frac 12 \dd t\la g, i A_0(\lambda)g\ra$, up to terms of order $\ll |t|^{-\frac{N-5}{N-6}}$.
In this proof, the sign $\simeq$ will mean ``up to terms of order $\ll |t|^{-\frac{N-5}{N-6}}$ as $|T_0| \to +\infty$''.

Since $iA_0(\lambda)$ is symmetric, we have
\begin{equation}
  \label{eq:deriv-correction}
  \frac 12 \dd t \la g, i A_0(\lambda) g\ra = \frac 12 \lambda'\la g, i \partial_\lambda A_0(\lambda)g\ra + \la \partial_t g, i A_0(\lambda) g\ra.
\end{equation}
The first term is of size $\lesssim \big|\frac{\lambda'}{\lambda}\big|\cdot \|g\|_\cE^2 \ll |t|^{-\frac{N-5}{N-6}}$, hence negligible.
We expand $\partial_t g$ according to \eqref{eq:dtg}. Consider the terms in the second line of \eqref{eq:dtg}. It follows from \eqref{eq:param-all}
and the fact that $\|A_0(\lambda)g\|_{\dot H^{-1}} \lesssim \|g\|_\cE$ that their contribution is $\lesssim |t|^{-1}\|g\|_\cE \leq |t|^{-\frac{3N-13}{2(N-6)}} \ll |t|^{-\frac{N-5}{N-6}}$,
hence negligible, so we can write
\begin{equation}
  \label{eq:deriv-correction-1}
  \frac 12 \dd t \la g, i A_0(\lambda) g\ra \simeq \la \Delta g + f(\eee^{i\zeta}W_\mu + \eee^{i\theta}W_\lambda + g) - f(\eee^{i\zeta}W_\mu) - f(\eee^{i\theta}W_\lambda), A_0(\lambda) g\ra.
\end{equation}
We now check that
\begin{equation}
  \label{eq:deriv-correction-2}
  |\la f(\eee^{i\zeta}W_\mu + \eee^{i\theta}W_\lambda) - f(\eee^{i\zeta}W_\mu) - f(\eee^{i\theta}W_\lambda), A_0(\lambda)g\ra| \ll |t|^{-\frac{N-5}{N-6}}.
\end{equation}
The function $A_0(\lambda)g$ is supported in the ball of radius $\wt R\lambda$. In this region we have $W_\lambda \ll W_\mu$, hence \eqref{eq:pointwise-2} yields
$|\la f(\eee^{i\zeta}W_\mu + \eee^{i\theta}W_\lambda) - f(\eee^{i\zeta}W_\mu) - f(\eee^{i\theta}W_\lambda)| \lesssim |W_\lambda|^\frac{4}{N-2}$.
By a change of variable we obtain
\begin{equation}
  \|W_\lambda^\frac{4}{N-2}\|_{L^2(|x| \leq \wt R\lambda)} = \lambda^\frac{N-2}{2}\|W^\frac{4}{N-2}\|_{L^2(|x| \leq \wt R)} \lesssim |t|^{-\frac{N-2}{N-6}}.
\end{equation}
By the first property in Lemma~\ref{lem:op-A}, there holds $\|A_0(\lambda)g\|_{L^2} \lesssim \lambda^{-1}\|g\|_\cE \lesssim |t|^{-\frac{N-5}{2(N-6)}}$,
hence the Cauchy-Schwarz inequality implies \eqref{eq:deriv-correction-2} (with a large margin).
By the triangle inequality, \eqref{eq:deriv-correction-1} and \eqref{eq:deriv-correction-2} yield
\begin{equation}
  \label{eq:deriv-correction-3}
  \frac 12 \dd t \la g, i A_0(\lambda) g\ra \simeq \la \Delta g + f(\eee^{i\zeta}W_\mu + \eee^{i\theta}W_\lambda + g) - f(\eee^{i\zeta}W_\mu + \eee^{i\theta}W_\lambda), A_0(\lambda) g\ra.
\end{equation}
We transform the right hand side using \eqref{eq:A-by-parts}, \eqref{eq:A-pohozaev} and the fact that $A_0(\lambda)g = \frac{1}{N\lambda^2} \Delta q\big(\frac{\cdot}{\lambda}\big)g + A(\lambda)g$.
Note that for any $c_2 > 0$ we have $\frac{c_0}{\lambda^2}\|g\|_\cE^2 \leq \frac{c_2}{2}|t|^{-\frac{N-5}{N-6}}$ if we choose $c_0$ small enough, thus
\begin{equation}
  \label{eq:deriv-correction-expand}
  \begin{aligned}
    &\frac 12 \dd t \la g, i A_0(\lambda) g\ra \leq c_2|t|^{-\frac{N-5}{N-6}} \\
    &-\frac{1}{\lambda^2}\Big(\int_{|x| \leq R\lambda}|\grad g|^2 \ud x - \big\la f(\eee^{i\zeta}W_\mu + \eee^{i\theta}W_\lambda + g) - f(\eee^{i\zeta}W_\mu + \eee^{i\theta}W_\lambda), \frac 1N \Delta q\big(\frac{\cdot}{\lambda}\big)g\big\ra\Big) \\
    &- \la A(\lambda)(\eee^{i\zeta}W_\mu + \eee^{i\theta}W_\lambda), f(\eee^{i\zeta}W_\mu + \eee^{i\theta}W_\lambda + g) - f(\eee^{i\zeta}W_\mu + \eee^{i\theta}W_\lambda) - f'(\eee^{i\zeta}W_\mu + \eee^{i\theta}W_\lambda)g\ra,
\end{aligned}
\end{equation}
where $c_2$ can be made arbitrarily small. Consider the second line. We will check that
\begin{equation}
  \label{eq:deriv-correction-4}
  \Big|\big\la f(\eee^{i\zeta}W_\mu + \eee^{i\theta}W_\lambda + g) - f(\eee^{i\zeta}W_\mu + \eee^{i\theta}W_\lambda), \frac 1N \Delta q\big(\frac{\cdot}{\lambda}\big)g\big\ra - \la f'(\eee^{i\theta}W_\lambda)g, g\ra\Big| \ll |t|^{-\frac{N-1}{N-6}}.
\end{equation}
Indeed, $\Delta q$ is bounded, hence $\big\|\frac 1N \Delta q\big(\frac{\cdot}{\lambda}\big)g\big\|_{L^\frac{2N}{N-2}} \lesssim \|g\|_\cE$.
By \eqref{eq:pointwise-1} we have
\begin{equation}
  \|f(\eee^{i\zeta}W_\mu + \eee^{i\theta}W_\lambda + g) - f(\eee^{i\zeta}W_\mu + \eee^{i\theta}W_\lambda) - f'(\eee^{i\zeta}W_\mu + \eee^{i\theta}W_\lambda)g\|_{L^\frac{2N}{N+2}} \lesssim \|g\|_\cE^\frac{N+2}{N-2} \ll \|g\|_\cE.
\end{equation}
Now from \eqref{eq:pointwise-5} we obtain
\begin{equation}
  \big\|\big(f'(\eee^{i\zeta}W_\mu + \eee^{i\theta}W_\lambda)- f'(\eee^{i\theta}W_\lambda)\big)g\big\|_{L^\frac{2N}{N+2}(|x| \leq \wt R\lambda)}
  \lesssim \|f'(\eee^{i\zeta}W_\mu)\|_{L^\frac N2(|x| \leq \wt R\lambda)}\|g\|_\cE \ll \|g\|_\cE.
\end{equation}
We have obtained
\begin{equation}
  \Big|\big\la f(\eee^{i\zeta}W_\mu + \eee^{i\theta}W_\lambda + g) - f(\eee^{i\zeta}W_\mu + \eee^{i\theta}W_\lambda), \frac 1N \Delta q\big(\frac{\cdot}{\lambda}\big)g\big\ra - \big\la f'(\eee^{i\theta}W_\lambda)g, \frac 1N \Delta q\big(\frac{\cdot}{\lambda}\big)g\big\ra\Big| \ll |t|^{-\frac{N-1}{N-6}}.
\end{equation}
But $\frac 1N \Delta q\big(\frac{x}{\lambda}\big) = 1$ for $|x| \leq R\lambda$ and $\|f'(\eee^{i\theta}W_\lambda)\|_{L^\frac N2(|x| \geq R\lambda)} \ll 1$ for $R$ large. This proves \eqref{eq:deriv-correction-4}.

The bounds \eqref{eq:bootstrap-unstable} and \eqref{eq:bootstrap-stable} together with \eqref{eq:coer-L-2} imply that
\begin{equation}
  \int_{|x| \leq R\lambda}|\grad g|^2 \ud x - \la f'(\eee^{i\theta}W_\lambda)\big)g, g\ra \geq -c_3\|g\|_\cE^2,
\end{equation}
with $c_3$ as small as we like by enlarging $R$. Thus, we have obtained that the second line in \eqref{eq:deriv-correction-expand} is $\leq c_2|t|^{-\frac{N-5}{N-6}}$,
with $c_2$ which can be made arbitrarily small.

We are left with the third line of \eqref{eq:deriv-correction-expand}. We will show that it equals $\frac{1}{\lambda^2}K$ up to negligible terms.
The support of $A(\lambda)(\eee^{i\zeta}W_\mu)$ is contained in $|x| \leq \wt R\lambda$ and $\|A(\lambda)(\eee^{i\zeta}W_\mu)\|_{L^\infty} \lesssim \lambda^{-2}$,
hence
\begin{equation}
  \|A(\lambda)(\eee^{i\zeta}W_\mu)\|_{L^\frac{2N}{N-2}} \lesssim \big(\lambda^N\lambda^{-\frac{4N}{N-2}}\big)^\frac{N-2}{2N} = \lambda^\frac{N-6}{2} \sim |t|^{-1}.
\end{equation}
From \eqref{eq:pointwise-1} and H\"older we have
\begin{equation}
  \|f(\eee^{i\zeta}W_\mu + \eee^{i\theta}W_\lambda + g) - f(\eee^{i\zeta}W_\mu + \eee^{i\theta}W_\lambda) - f'(\eee^{i\zeta}W_\mu + \eee^{i\theta}W_\lambda)g\|_{L^\frac{2N}{N+2}} \lesssim \|g\|_\cE^\frac{N+2}{N-2} \ll |t|^{-\frac{1}{N-6}}.
\end{equation}
Thus, in the third line of \eqref{eq:deriv-correction-expand} we can replace $A(\lambda)(\eee^{i\zeta}W_\mu + \eee^{i\theta}W_\lambda)$ by $A(\lambda)(\eee^{i\theta}W_\lambda)$.
Property \ref{enum:gradlap} implies that $|AW - \Lambda W| \lesssim W$ pointwise, with a constant independent of $c$ and $R$ used in the definition of the function $q$.
After rescaling and phase change we obtain $\big|A(\lambda)(\eee^{i\theta}W_\lambda) - \frac{1}{\lambda^2}\eee^{i\theta}\Lambda W_\lambda\big| \lesssim \frac{1}{\lambda^2} W_\lambda$.
But $A(\lambda)W = \frac{1}{\lambda^2}\Lambda W_\lambda$ for $|x| \leq R\lambda$, so we obtain
\begin{equation}
  \begin{aligned}
  &\Big|\la A(\lambda)(\eee^{i\theta}W_\lambda) - \frac{1}{\lambda^2}\eee^{i\theta}\Lambda W_\lambda, f(\eee^{i\zeta}W_\mu + \eee^{i\theta}W_\lambda + g) - f(\eee^{i\zeta}W_\mu + \eee^{i\theta}W_\lambda) - f'(\eee^{i\zeta}W_\mu + \eee^{i\theta}W_\lambda)g\ra\Big| \\
  & \lesssim \frac{1}{\lambda^2}\int_{|x| \geq R\lambda}W_\lambda\cdot |f(\eee^{i\zeta}W_\mu + \eee^{i\theta}W_\lambda + g) - f(\eee^{i\zeta}W_\mu + \eee^{i\theta}W_\lambda) - f'(\eee^{i\zeta}W_\mu + \eee^{i\theta}W_\lambda)g|\ud x.
  \end{aligned}
\end{equation}
Since $|\zeta - \theta| \simeq \frac{\pi}{2}$, we have $|\eee^{i\zeta}W_\mu + \eee^{i\theta}W_\lambda| \gtrsim W_\lambda$, hence \eqref{eq:pointwise-1} yields
\begin{equation}
  W_\lambda\cdot |f(\eee^{i\zeta}W_\mu + \eee^{i\theta}W_\lambda + g) - f(\eee^{i\zeta}W_\mu + \eee^{i\theta}W_\lambda) - f'(\eee^{i\zeta}W_\mu + \eee^{i\theta}W_\lambda)g| \lesssim W_\lambda^\frac{4}{N-2}|g|^2.
\end{equation}
Integrating over $|x| \geq R\lambda$ and using H\"older we find
\begin{equation}
    \begin{aligned}
  &\Big|\la A(\lambda)(\eee^{i\theta}W_\lambda) - \frac{1}{\lambda^2}\eee^{i\theta}\Lambda W_\lambda, f(\eee^{i\zeta}W_\mu + \eee^{i\theta}W_\lambda + g) - f(\eee^{i\zeta}W_\mu + \eee^{i\theta}W_\lambda) - f'(\eee^{i\zeta}W_\mu + \eee^{i\theta}W_\lambda)g\ra\Big| \\
  & \lesssim c_2 |t|^{-\frac{N-5}{N-6}},\qquad \text{with }c_2 \text{ arbitrarily small as }R \to +\infty.
\end{aligned}
\end{equation}

Resuming all the computations starting with \eqref{eq:deriv-correction}, we have shown that
\begin{equation}
  \frac 12 \dd t \la g, iA_0(\lambda) g\ra \leq \frac{c_1}{2}|t|^{-\frac{N-5}{N-6}} + \frac{1}{\lambda^2}K.
\end{equation}
Hence \eqref{eq:deriv-psi-1} yields \eqref{eq:deriv-psi}.

Since $\theta(T) = 0$, we have $|\theta(T)| \lesssim \|g(T)\|_\cE^2 \ll |T|^{-\frac{1}{N-6}}$.
Integrating \eqref{eq:deriv-psi} on $[T, t]$ we get $\psi(t) \gtrsim -c_1|t|^{-\frac{1}{N-6}}$. But $|\la g(t), A_0(\lambda)g(t)\ra| \lesssim \|g(t)\|_{\cE}^2 \leq |t|^{-\frac{N-1}{N-6}} \ll |t|^{-\frac{1}{N-6}}$, hence we obtain $\theta(t) \gtrsim -c_1|t|^{-\frac{1}{N-6}}$, which yields \eqref{eq:bootstrap-bbetter-theta-geq} if $c_1$ is chosen small enough.
This finishes the proof of \eqref{eq:bootstrap-better-theta}.

\textbf{Step 3.}
From \eqref{eq:coer-conclusion} we obtain $\|g\|_\cE^2 + C_0 \theta \lambda^\frac{N-2}{2} \leq C_1|t|^{-\frac{N}{N-6}}$, hence
\begin{equation}
  \|g\|_\cE^2 \leq -C_0 \theta \lambda^\frac{N-2}{2} + C_1|t|^{-\frac{N}{N-6}} \leq \frac 18|t|^{-\frac{N-1}{N-6}} + C_1|t|^{-\frac{N}{N-6}},
\end{equation}
provided that $c_0$ in \eqref{eq:bootstrap-bbetter-theta} is small enough. This yields \eqref{eq:bootstrap-better-g}.
\end{proof}

\subsection{Choice of the initial data by a topological argument}
The bootstrap in Proposition~\ref{prop:bootstrap} leaves out the control of $\lambda(t)$, $a_1^+(t)$ and $a_2^+(t)$.
We will tackle this problem here.

\begin{proposition}
  \label{prop:shooting}
  Let $|T_0|$ be large enough. For all $T < T_0$ there exist $\lambda^0, a_1^0, a_2^0$ satisfying \eqref{eq:initial-assum}
  such that the solution $u(t)$ with the initial data $u(T) = -iW + W_{\lambda^0} + g^0$ exists on the time interval $[T, T_0]$
  and for $t\in[T, T_0]$ the bounds \eqref{eq:bootstrap-better-zeta}, \eqref{eq:bootstrap-better-mu}, \eqref{eq:bootstrap-better-theta}, \eqref{eq:bootstrap-better-g},
  \begin{align}
    \big|\lambda(t) - \frac{1}{\kappa}(\kappa|t|)^{-\frac{2}{N-6}}\big| &\leq \frac 12 |t|^{-\frac{5}{2(N-6)}}, \label{eq:bootstrap-better-lambda} \\
    |a_1^+(t)| &\leq \frac 12 |t|^{-\frac{N}{2(N-6)}}, \label{eq:bootstrap-better-a1p} \\
    |a_2^+(t)| &\leq \frac 12 |t|^{-\frac{N}{2(N-6)}} \label{eq:bootstrap-better-a2p}
  \end{align}
  hold.
\end{proposition}
The proof will be split into some lemmas.
For $t \in [T, T_0]$, $\wt \lambda > 0$, $\wt a_1 \in \bR$ and $\wt a_2 \in \bR$ we denote
\begin{equation*}
  X_t(\wt \lambda, \wt a_1, \wt a_2) := \Big(\frac{1}{\kappa}(\kappa|t|)^{-\frac{2}{N-6}} + \wt \lambda|t|^{-\frac{5}{2(N-6)}},
  \wt a_1|t|^{-\frac{N}{2(N-6)}}, \wt a_2|t|^{-\frac{N}{2(N-6)}}\Big).
\end{equation*}
We see that $\lambda(t)$, $a_1^+(t)$ and $a_2^+(t)$ satisfy \eqref{eq:bootstrap-better-lambda}, \eqref{eq:bootstrap-better-a1p}
and \eqref{eq:bootstrap-better-a2p} if and only if
\begin{equation}
  X_t^{-1}(\lambda(t), a_1^+(t), a_2^+(t)) \in Q := \Big[{-}\frac 12, \frac 12\Big]^3.
\end{equation}
\begin{lemma}
  Assume that $\lambda(t)$, $a_1^+(t)$ and $a_2^+(t)$ satisfy \eqref{eq:mod-l},
  \eqref{eq:proper-1p} and \eqref{eq:proper-2p} on the time interval $t \in (T_1, T_2)$
  and that
  \begin{equation}
    (p_0, p_1, p_2) := X_t^{-1}(\lambda(t), a_1^+(t), a_2^+(t)) \in Q \setminus \partial Q\qquad \text{for all }t \in (T_1, T_2).
  \end{equation}
  Then for all $t \in (T_1, T_2)$ there holds
  \begin{align}
    \Big|p_0'(t) - \frac{2N - 13}{2(N-6)}|t|^{-1}p_0(t)\Big| &\leq c|t|^{-1}, \label{eq:cube-p0} \\
    \Big|p_1'(t) - \frac{\nu}{\mu(t)}p_1(t)\Big| &\leq \frac{c}{\mu(t)}, \label{eq:cube-p1} \\
    \Big|p_2'(t) - \frac{\nu}{\lambda(t)}p_2(t)\Big| &\leq \frac{c}{\lambda(t)}, \label{eq:cube-p2} \\
  \end{align}
  where $c > 0$ can be made arbitrarily small by taking $T_0$ large enough.
\end{lemma}
\begin{proof}
  By the definition of $p_0(t)$ we have
  \begin{equation}
    \label{eq:lambda-p0}
    \lambda = \frac{1}{\kappa}(\kappa|t|)^{-\frac{2}{N-6}} + p_0(t)|t|^{-\frac{5}{2(N-6)}}.
  \end{equation}
  Differentiating in time we obtain
  \begin{equation}
    \lambda'(t) = \frac{2}{N-6}(\kappa|t|)^{-\frac{N-4}{N-6}} + \frac{5}{2(N-6)}|t|^{-\frac{2N-7}{2(N-6)}}p_0(t)
    + |t|^{-\frac{5}{2(N-6)}}p_0(t).
  \end{equation}
  Applying the Newton formula to \eqref{eq:lambda-p0} and using the fact that $|p_0| \lesssim 1$ we get
  \begin{equation}
    \lambda^\frac{N-4}{2} \kappa^{-\frac{N-4}{2}}(\kappa|t|)^{-\frac{N-4}{N-6}}
    + \frac{N-4}{2}\kappa^{-\frac{N-6}{2}}(\kappa|t|)^{-1}p_0(t)|t|^{-\frac{5}{2(N-6)}} + O(|t|^{-\frac{N-3}{N-6}}),
  \end{equation}
  thus
  \begin{equation}
    \lambda' - \frac{2\kappa^\frac{N-4}{2}}{N-6}\lambda^\frac{N-4}{2} = \Big(\frac{5}{2(N-6)} - \frac{N-4}{N-6}\Big)|t|^{-\frac{2N-7}{2(N-6)}}p_0(t) + |t|^{-\frac{5}{2(N-6)}}p_0'(t) + O(|t|^{-\frac{N-3}{N-6}}).
  \end{equation}
  Using \eqref{eq:mod-l} and multiplying both sides by $|t|^\frac{5}{2(N-6)}$ we obtain \eqref{eq:cube-p0}.

  We have $a_1^+(t) = |t|^{-\frac{N}{2(N-6)}}p_1(t)$, which yields
  \begin{equation}
    \dd t a_1^+ - \frac{\nu}{\mu}a_1^+ = |t|^{-\frac{N}{2(N-6)}}\big(p_1'(t) - \frac{\nu}{\mu}p_1(t)\big) + O\big(|t|^{-\frac{N}{2(N-6)} - 1}\big),
  \end{equation}
  so \eqref{eq:proper-1p} implies \eqref{eq:cube-p1}. The proof of \eqref{eq:cube-p2} is similar.
\end{proof}
For $C > 1$, $j \in \{0, 1, 2\}$ and $p \in \bR^3$ we denote
\begin{equation}
  V_j^C(p) := \{p + (r_0, r_1, r_2): \sign(r_j) = \sign(p_j)\text{ and }\max_j|r_j| < C|r_j|\}.
\end{equation}
\begin{lemma}
  \label{lem:top-2}
  Assume that $\lambda(t)$, $a_1^+(t)$ and $a_2^+(t)$ satisfy \eqref{eq:bootstrap-lambda},
  \eqref{eq:mod-l}, \eqref{eq:bootstrap-unstable}, \eqref{eq:proper-1p} and \eqref{eq:proper-2p}
  for $t \in (T_1, T_2)$.
  There exists a constant $C > 0$, depending on $T_1$ and $T_2$,
  such that if for some $T_3 \in (T_1, T_2)$ and $j \in \{0, 1, 2\}$ there holds
  $|p_j(T_3)| \geq \frac 14$, then for all $t \in (T_3, T_2)$ there holds $p(t) \in V_j^C(p(T_3))$.
\end{lemma}
\begin{proof}
From the previous lemma we infer that there exist strictly positive constants $c_1$ and $C_1$, depending on $T_1$ and $T_2$,
such that $|p_j'(t)| \leq C_1$ and
\begin{equation}
  |p_j(t)| \geq \frac 14\quad \Rightarrow\quad |p_j'(t)| \geq c_1\text{ and }\sign p_j'(t) = \sign p_j(t).
\end{equation}
It is sufficient to take $C > \frac{C_1}{c_1}$.
\end{proof}
\begin{proof}
  The proof proceeds by contradiction. Supposing that the result does not hold, we will construct a continuous retraction $\Phi : Q \to \partial Q$, $\Phi(p) = p$ for $p \in \partial Q$. It is a well-known fact from topology that such a function $\Phi$ does not exist.

  Let $p^0 \in Q$. Take $(\lambda^0, \wt a_1^0, \wt a_2^0) = X_{T}(p^0)$ and let $g^0$ be given by Lemma~\ref{lem:initial}.
  Let $u: [T, T_+) \to \cE$ be the solution of \eqref{eq:nls} for the initial data $u(T) = -iW + W_{\lambda^0} + g^0$.
We will say that the solution $u$ is associated with $p^0 \in Q$.

  Let $T_2$ be the infimum of the values of $t \in [T, T_+)$ such that \eqref{eq:bootstrap-better-zeta}, \eqref{eq:bootstrap-better-mu}, \eqref{eq:bootstrap-better-theta}, \eqref{eq:bootstrap-better-g},
    \eqref{eq:bootstrap-better-lambda}, \eqref{eq:bootstrap-better-a1p} or \eqref{eq:bootstrap-better-a2p} does not hold.
  By our assumption that Proposition~\ref{prop:shooting} is false, we have that $T_2$ exists and $T_2 < T_0$.
  Indeed, if all the listed conditions were satisfied for $t \in [T, T_+)$, then Corollary~\ref{cor:leaves-compact} would imply that $T_+ > T_0$,
    hence all the conditions would hold on $[T, T_0]$, which contradicts the assumption.

    Set $p^1 := X_{T_2}^{-1}(\lambda(T_2), a_1^+(T_2), a_2^+(T_2))$.
  By continuity $p^1 \in Q$, and we will show that in fact $p^1 \in \partial Q$.
  Indeed, by continuity of the flow, the assumptions of Proposition~\ref{prop:bootstrap} are satisfied for $T_1 = T_2 + \tau$ for some $\tau > 0$.
  Hence \eqref{eq:bootstrap-better-zeta}, \eqref{eq:bootstrap-better-mu}, \eqref{eq:bootstrap-better-theta} and \eqref{eq:bootstrap-better-g}
  continue to hold on $[T_2, T_2 + \tau]$, so one of the conditions \eqref{eq:bootstrap-better-lambda}, \eqref{eq:bootstrap-better-a1p} or \eqref{eq:bootstrap-better-a2p}
  is violated somewhere on $[T_2, T_2 + \tau]$ for every $\tau > 0$. By continuity of the parameters with respect to time,
  this yields $p^1 \in \partial Q$.

  We set
  \begin{equation}
    \Phi: Q \to \partial Q,\qquad \Phi(p^0) := p^1.
  \end{equation}
  It is immediate from the definition that $\Phi(p) = p$ for $p \in \partial Q$, and it remains to show that $\Phi$ is continuous.
 
  Let $p^0 \in Q$, $\Phi(p^0) = p^1 \in \partial Q$ and $\varepsilon > 0$.
  Let $C$ be the constant from Lemma~\ref{lem:top-2} for $T_1 = T$ and $T_2 = T_0$.
  We will consider the case $p_0^1 = \frac 12$, the other cases being similar.
  It is clear that for $\delta > 0$ small enough $V_\delta := V_0^C\big(\frac 12 - \delta, p_1^1, p_2^1\big) \cap \partial Q$
  is an $\varepsilon$-neighborhood of $p^1$.
  Thus, by Lemma~\ref{lem:top-2}, in order to finish the proof it suffices to show that
  if $q^0 \in Q$ with $|q^0 - p^0|$ small enough, then the solution associated with $q$ passes through $V_\delta$.

  If $p^0 = p^1 \in \partial Q$, this is obvious, since $V_\delta$ is in this case a neighborhood of $p^0$.
  In the case $p^0 \in Q \setminus \partial Q$, the solution associated with $p^0$ passes through $V_\delta$
  before reaching $\partial Q$. Thus, by the continuous dependence on the initial data,
  the solution associated with $q^0$ passes through $V_\delta$ if $|q^0 - p^0|$ is small enough.
\end{proof}

\begin{proof}[Proof of Theorem~\ref{thm:deux-bulles}]
  Let $T_0 < 0$ be given by Proposition~\ref{prop:shooting} and let $T_0, T_1, T_2, \ldots$
  be a decreasing sequence tending to $-\infty$.
  For $n \geq 1$, let $u_n$ be the solution given by Proposition~\ref{prop:shooting}.
  Inequalities \eqref{eq:bootstrap-better-zeta}, \eqref{eq:bootstrap-better-mu}, \eqref{eq:bootstrap-better-theta},
  \eqref{eq:bootstrap-better-lambda} and \eqref{eq:bootstrap-better-g} yield
  \begin{equation}
    \label{eq:uniform}
    \Big\|u_n(t) - \Big({-}iW + W_{\frac{1}{\kappa}(\kappa|t|)^{-\frac{2}{N-6}}}\Big)\Big\|_\cE \lesssim |t|^{-\frac{1}{2(N-6)}},
  \end{equation}
  for all $t \in [T_n, T_0]$ and with a constant independent of $n$.
  Upon passing to a subsequence, we can assume that $u_n(T_0) \wto u_0 \in \cE$.
  Let $u$ be the solution of \eqref{eq:nls} with the initial condition $u(T_0) = u_0$.
Corollary~\ref{cor:weak-cont} implies that $u$ exists on the time interval $({-}\infty, T_0]$
and for all $t \in ({-}\infty, T_0]$ there holds $u_n(t) \wto u(t)$.
Passing to the weak limit in \eqref{eq:uniform} finishes the proof.
\end{proof}

\appendix

\section{Cauchy theory}
\label{sec:cauchy}

\subsection{Profile decomposition}
We recall briefly the profile decomposition method of Bahouri and G\'erard \cite{BaGe99}, and Merle and Vega \cite{MeVe98}.
In the case of the energy-critical defocusing NLS, the corresponding theory was developped by Keraani \cite{Ker01}.
For the focusing NLS in high dimensions, which is the case discussed in this paper, see Killip and Visan \cite{KiVi10}.

\begin{proposition}[Killip, Visan]
Let $u_{0,n}$ be a bounded sequence in $\cE$. There exists a subsequence of $u_{0,n}$, still denoted $u_{0,n}$,
such that there exist a family of solutions of the linear Schr\"odinger equation $U\lin^j(t) = \eee^{it\Delta}U_0^j$
and a family of sequences of parameters $t_n^j$ and $\lambda_n^j$ satisfying the pseudo-orthogonality condition
\begin{equation}
j \neq k \ \Rightarrow \ \lim_{n \to +\infty} \frac{\lambda_n^j}{\lambda_n^k} + \frac{\lambda_n^k}{\lambda_n^j} + \frac{|t_n^j - t_n^k|}{\lambda_n^j} = +\infty
\end{equation}
such that for all $J \geq 0$
\begin{equation}
\label{eq:lin-profils}
u_{0,n} = \sum_{j=1}^J U\lin^j\Big(\frac{-t_n^j}{\lambda_n^j}\Big)_{\lambda_n^j} + w_n^J,
\end{equation}
with
\begin{equation}
\lim_{J\to +\infty} \limsup_{n\to +\infty}\|\eee^{it\Delta}w_n^J\|_{L_{t, x}^\frac{2(N+2)}{N-2}} = 0.
\end{equation}
Moreover, for any $J \geq 0$ there holds
\begin{equation}
\lim_{n \to +\infty} \Big|\|u_{0,n}\|_\cE^2 - \sum_{j=1}^J \|U_0^j\|_\cE^2 - \|w_n^J\|_\cE^2\Big| = 0.
\end{equation} \qed
\end{proposition}
Formula \eqref{eq:lin-profils} is called the linear profile decomposition. In the applications, we regard $u_{0,n}$ as a sequence
of initial data of solutions $u_n$ of \eqref{eq:nls}. In order to approximate the solutions $u_n$, we introduce the nonlinear profiles.
The nonlinear profile $U^j$ corresponding to the linear profile $U\lin^j$ is defined as the solution of \eqref{eq:nls}
such that
\begin{equation}
\lim_{n\to+\infty}\Big\|U^j\Big(\frac{-t_n^j}{\lambda_n^j}\Big)- U\lin^j\Big(\frac{-t_n^j}{\lambda_n^j}\Big)\Big\|_\cE = 0.
\end{equation}

The next proposition is a version of the result of Keraani for the focusing NLS.
Its statement is very similar to Proposition~2.8 in \cite{DKM1}.
\begin{proposition}
\label{prop:profils}
Let $u_{0,n}$ be a sequence in $\cE$ with a linear profile decomposition \eqref{eq:lin-profils} and let $U^j:(T_-(U^j), T_+(U^j)) \to \cE$ be the nonlinear profiles.
Let $\tau_n > 0$ be a sequence such that for all $j$ and $n$
\begin{equation}
\frac{\tau_n - t_n^j}{(\lambda_n^j)^2} < T_+(U^j),\qquad \limsup_{n\to+\infty}\|U^j\|_{L^\frac{2(N+2)}{N-2}\big(\big[\frac{-t_n^j}{(\lambda_n^j)^2}, \frac{\tau_n - t_n^j}{(\lambda_n^j)^2}\big]\times \bR^N\big)} < +\infty.
\end{equation}
Let $u_n$ be the solution of \eqref{eq:nls} with the initial data $u_n(0) = u_{0,n}$. Then, for $n$ large, $u_n$ exists on the time interval $[0, \tau_n]$, $\limsup_{n\to+\infty}\|u_n\|_{L^\frac{2(N+2)}{N-2}([0, \tau_n]\times \bR^N)} < +\infty$ and for all $J \geq 0$
\begin{equation}
u_n(t) = \sum_{j=1}^J U^j\Big(\frac{t-t_n^j}{(\lambda_n^j)^2}\Big)_{\lambda_n^j} + w_n^J(t) + r_n^J(t),
\end{equation}
with
\begin{equation}
\lim_{J\to +\infty}\limsup_{n\to+\infty} \Big(\|r_n^J\|_{L^\frac{2(N+2)}{N-2}([0, \tau_n]\times \bR^N)} + \sup_{t\in[0, \tau_n]}\|r_n^J\|_\cE\Big) = 0.
\end{equation}
\end{proposition}
\begin{proof}
See \cite{DKM1}, proof of Proposition~2.8, and \cite{KiVi10}, proof of Lemma 3.2.
\end{proof}
\subsection{Corollaries}
\begin{corollary}
  \label{cor:leaves-compact}
  There exists a constant $\eta > 0$ such that the following holds. Let $u: [t_0, T_+) \to \cE$
    be a maximal solution of \eqref{eq:nls} with $T_+ < +\infty$. Then for any compact set $K \subset \cE$
    there exists $\tau < T_+$ such that $\dist(u(t), K) > \eta$ for $t \in [\tau, T_+)$.
\end{corollary}
\begin{proof}
See \cite{moi16p}, Corollary A.4.
\end{proof}
\begin{corollary}
  \label{cor:weak-cont}
  There exists a constant $\eta > 0$ such that the following holds.
  Let $K \subset \cE$ be a compact set and let $u_n: [T_1, T_2] \to \cE$ be a sequence of solutions of \eqref{eq:nls} such that
  \begin{equation}
    \dist(u_n(t), K) \leq \eta,\qquad \text{for all }n \in \bN\text{ and }t \in [T_1, T_2].
  \end{equation}
  Suppose that $u_n(T_1) \wto u_0 \in \cE$. Then the solution $u(t)$ of \eqref{eq:nls} with the initial condition $u(T_1) = u_0$
  is defined for $t \in [T_1, T_2]$ and
  \begin{equation}
    u_n(t) \wto u(t),\qquad \text{for all }t \in [T_1, T_2].
  \end{equation}
\end{corollary}
\begin{proof}
See \cite{moi16p}, Corollary A.6.
\end{proof}

\bibliographystyle{plain}
\bibliography{nls-deux-bulles}

\end{document}